\documentclass[reqno,oneside,11pt]{amsart}

\usepackage{amssymb,mathptmx,eucal,array,enumitem,setspace,color,multirow}
\usepackage{palatino}

\usepackage[T1]{fontenc}

\usepackage{mathtools} 
\usepackage[noadjust]{cite}
\usepackage{geometry}

\usepackage{tikz}
\usetikzlibrary{decorations.pathmorphing,cd} 
\tikzset{>=to} 

\definecolor{rouge}{rgb}{0.95,0.1,0.15}
\definecolor{forestgreen}{rgb}{0.13,0.54,0.13}
\definecolor{bleu}{rgb}{0.1,0.1,0.9}
\newcommand{\modif}[1]{#1}

\usepackage[pdfpagemode=UseNone, pdfstartview={XYZ null null null}]{hyperref}
\hypersetup{
colorlinks=true,
citecolor=red,
linkcolor=blue,
urlcolor=blue
}
\usepackage[capitalise,noabbrev]{cleveref} 

\setlist[itemize]{leftmargin=*}                  
\setlist[enumerate]{leftmargin=*,                
                    label=\textup{(\roman*)}}    

\DeclareSymbolFont{largesymbols}{OMX}{zplm}{m}{n} 

\let\originalleft\left     
\let\originalright\right
\renewcommand{\left}{\mathopen{}\mathclose\bgroup\originalleft}
\renewcommand{\right}{\aftergroup\egroup\originalright}
\usepackage{marginnote}
\numberwithin{equation}{section}

\newcolumntype{C}{>{$}c<{$}}              
\newcolumntype{D}{>{$\displaystyle}l<{$}} 
\newcommand{\fld}[1]{\mathbb{#1}} 
\newcommand{\ZZ}{\fld{Z}}
\newcommand{\NN}{\fld{N}}

\newcommand{\RR}{\fld{R}}
\newcommand{\CC}{\fld{C}}

\newcommand{\dd}{\mathrm{d}}

\newcommand{\alg}[1]{\mathcal{#1}}  

\newcommand{\modSB}{~\operatorname{mod}\,} 

\theoremstyle{plain}
\newtheorem{theorem}{Theorem}[section]
\newtheorem{lemma}[theorem]{Lemma}
\newtheorem{proposition}[theorem]{Proposition}
\newtheorem{corollary}[theorem]{Corollary}
\newtheorem{definition}[theorem]{Definition}

\newcommand{\bili}[2]{\left\langle #1,#2\right\rangle}
\newcommand{\sym}[1]{O_{#1}} 
\newcommand{\tot}[1]{L_{#1}} 
\newcommand{\comm}[1]{\left[ #1 \right]} 
\newcommand{\acomm}[1]{\left\lbrace #1 \right\rbrace} 
\newcommand{\dcover}[1]{\widetilde{#1}}
\newcommand{\Lad}[1]{L_{#1}} 
\newcommand{\Dun}[1]{\mathcal{D}_{#1}}
\newcommand{\Clif}{Cl}
\newcommand{\conj}[1]{\overline{#1}}
\newcommand{\Lie}[1]{\mathfrak{#1}} 
\newcommand{\osp}{\Lie{osp}}
\newcommand{\DDop}{\underline{D}} 
\newcommand{\xun}{\underline{x}} 
\newcommand{\Euler}{\mathbb{E}}
\newcommand{\Lapl}{\Delta}
\newcommand{\DLapl}{\Delta_{\kappa}}
\newcommand{\Asym}{\mathfrak{SA}}

\newcommand{\CK}{\mathbf{CK}}
\newcommand{\Poly}{\mathcal{P}}
\newcommand{\Mono}{\mathcal{M}}
\newcommand{\Harmo}{\mathcal{H}}

\DeclareMathOperator{\Pin}{Pin}

\newcommand{\Ortho}{\mathcal{O}}
\newcommand{\SymAlg}{\mathfrak{SA}_m}
\newcommand{\sumx}{\mathbf{x}^2}
\newcommand{\monopo}{\Phi}

\newcommand{\smallmat}[1]{\left( \begin{smallmatrix}
		#1
\end{smallmatrix}\right)} 

\newcommand{\dsig}[1]{\dcover{\sigma}_{#1}}
\newcommand{\dtau}{\dcover{\tau}} 

\newcommand{\hyperF}[5]{{}_{#1}F_{#2} \left({#3 \atop #4}\ \middle| \ #5 \right)}

\newcommand{\stdquo}{L_{\lambda,\Lambda}(V)}

\begin{document}

\title[Representations of the dihedral Dunkl--Dirac symmetry algebra]{Finite-dimensional representations of the symmetry algebra of the dihedral Dunkl--Dirac operator}

\author[H De Bie]{Hendrik De Bie}

\address[Hendrik De Bie]{ Clifford Research Group,
Department of Electronics and Information Systems,
Faculty of Engineering and Architecture,
Ghent University,
Krijgslaan 281--S8, 9000 Gent, Belgium}
\email{Hendrik.DeBie@UGent.be}

\author[A Langlois-R\'emillard]{Alexis Langlois-R\'emillard}

\address[Alexis Langlois-R\'emillard]{ Department of Applied Mathematics, Computer Science and Statistics, Faculty of Sciences, Ghent University, Krijgslaan 281--S9, 9000 Gent, Belgium.
}
\email{Alexis.LangloisRemillard@UGent.be}

\author[R Oste]{Roy Oste}

\address[Roy Oste]{ Department of Applied Mathematics, Computer Science and Statistics, Faculty of Sciences, Ghent University, Krijgslaan 281--S9, 9000 Gent, Belgium.
}
\email{Roy.Oste@UGent.be}

\author[J Van der Jeugt]{Joris Van der Jeugt}

\address[Joris Van der Jeugt]{ Department of Applied Mathematics, Computer Science and Statistics, Faculty of Sciences, Ghent University, Krijgslaan 281--S9, 9000 Gent, Belgium.}
\email{Joris.VanderJeugt@UGent.be}

\date{\today}



\begin{abstract}

The Dunkl--Dirac operator is a deformation of the Dirac operator by means of Dunkl derivatives. We investigate the symmetry algebra generated by the elements supercommuting with the Dunkl--Dirac operator and its dual symbol. This symmetry algebra is realised inside the tensor product of a Clifford algebra and a rational Cherednik algebra associated with a reflection group or root system. For reducible root systems of rank three, we determine all the irreducible finite-dimensional representations and conditions for unitarity. Polynomial solutions of the Dunkl--Dirac equation are given as a realisation of one family of such irreducible unitary representations.
\end{abstract}

\maketitle

\tableofcontents

\onehalfspacing

%
%
\section{Introduction} \label{sec:Intro}

The aim of this work is to continue the inquiry of the representation theory of the symmetry algebra, generated by symmetries supercommuting with an $\osp(1|2)$ realisation linked to a Dirac-like operator deformed by a reflection group~\cite{de_bie_algebra_2018}. We consider the case of a Dunkl--Dirac operator acting on a three-dimensional space as this is the lowest dimension where the symmetry algebra portrays interesting behaviour. This work enlarges the class of studied systems to all the reducible rank 3 root systems in a three-dimensional setting whereas only the cases $A_1\oplus A_1\oplus A_1$~\cite{de_bie_diracdunkl_2016}, $A_2$~\cite{debie_total_2018} and partially $G_2$~\cite{LROVarna19}  were known beforehand.

Dunkl operators are a generalisation of partial derivatives introduced by Dunkl~\cite{dunkl_differential-difference_1989} in the context of orthogonal polynomials in several variables. The constituents of those operators are: a reflection group $W\subset\mathcal{O}(N)$ acting on the Euclidean space $\RR^N$ and its root system $R$, a $W$-invariant function $\kappa : R\to \CC$, and the algebra of polynomials in $N$ variables $x_1,\dots , x_N$. Then the algebra $\alg{A}$ generated by the group $W$, the $N$ Dunkl operators $\Dun{1}, \dots,\Dun{N}$, and the $N$ multiplication operators $x_1,\dots x_N$ is a realisation of a rational double affine Hecke algebra, or rational Cherednik algebra~\cite{etingof_symplectic_2002}. The symmetry algebra we consider is the centralizer of an $\Lie{osp}(1|2)$ realisation inside the tensor product of a Clifford algebra and $\alg{A}$. The symmetry algebra is in fact the full centralizer of the $\osp(1|2)$ algebra realisation. This is considered in generality in work in preparation by the third author~\cite{roy_prep}.  It is related to the $(\operatorname{Pin}(N),\,\Lie{osp}(1|2))$ Howe duality. Other deformations of Howe dualities were recently studied in the rational Cherednik algebra context~\cite{ciubotaru_dunkl-cherednik_2018,ciubotaru_deformations_2020}. 

With a generalisation of derivatives, it is only natural to study the equivalent to various classical differential operators. We focus on the Dunkl--Dirac equation in three dimensions. It was first studied algebraically in the context of Bannai--Ito algebras~\cite{de_bie_diracdunkl_2016} for the $A_1\oplus A_1\oplus A_1$ root system and then the definition was extended to all root systems~\cite{de_bie_algebra_2018}. There are infinitely many root systems of rank 3, divided in three cases presented below. 
\begin{enumerate}
	\item Three rank 1 root systems: $A_1\oplus A_1\oplus A_1$, with Weyl group $\ZZ_2\times \ZZ_2 \times \ZZ_2$.
	\item The sum of a rank 1  and a rank 2 root systems, so the infinite family of dihedral root systems with an $A_1$ part: $A_1\oplus I_2(m)$ ($m\geq 3$). Their Coxeter group is $\ZZ_2\times D_{2m}$.
	\item The irreducible rank 3 root systems $A_3$, $B_3$ and $C_3$, of respective Weyl groups $S_4$ and $S_3\rtimes \ZZ_2^3$ for the two others.
\end{enumerate}
This article addresses the reducible root systems of rank 3 (cases (i) and (ii)) as they can all be studied using the same method. The study of the irreducible root systems $A_3$ and $B_3$ ($C_3$ has the same Weyl group and is indiscernible in our context) requires different methods and is subject to ongoing investigations.

Dihedral groups in the Dunkl operators context are well-studied examples. They offer a tractable non-trivial behaviour and already divide in two different cases \modif{whether the parameter $m$ for the dihedral group is odd or even.} Already in the original paper of Dunkl~\cite{dunkl_differential-difference_1989}, the harmonics of the dihedral Dunkl-Laplacian were given. The complete finite-dimensional representation theory in the dihedral case for rational Cherednik algebras is also known~\cite{chmutova_representations_2006}. They are often the first non-trivial examples one can hope to consider completely. Recent investigations on the dihedral case include: closed formulas for intertwining operators~\cite{xu_intertwining_2019,de_bie_lian_2020}, geometric properties of the Calogero--Moser space associated with dihedral groups~\cite{bonnafe_2018} and the complete descriptions of the deformed unitary Howe dual pairs~\cite{ciubotaru_deformations_2020}.  

Here we study the finite-dimensional irreducible representations of the symmetry algebra of the dihedral Dunkl--Dirac operator acting on a three-dimensional space. Albeit dihedral root systems are of rank 2, the symmetry algebra becomes interesting only for dimension three and higher. Adding the extra dimension by means of an $A_1$ root system is the most general approach. Note that the symmetry algebra associated with the reflection group $D_{2m}$ is a subalgebra of the one associated with $\ZZ_2\times D_{2m}$ and is not simply obtained by setting the function $\kappa$ to zero on the $\ZZ_2$ part. This is a difference with the study of the $S_3$ case~\cite{debie_total_2018}: the extra $\ZZ_2$ component adds more constraints on the forms the representations can take. Another difference is that restricting representations of the symmetry algebra to representations of $W$ now gives rise to projective representations of $W$. Projective representations are representations of the  double covering of the group. For a reflection group, considered as a subgroup of $\mathcal{O}(N)$, the two double coverings $\dcover{W}^+$ and $\dcover{W}^-$ can be viewed as the restrictions of the double coverings $\operatorname{Pin}_{+}(N)$ or $\operatorname{Pin}_{-}(N)$ of $\mathcal{O}(N)$. It was not apparent for the root system $A_2$ because the group $S_3$ does not admit non-trivial double coverings. We give the detailed construction of the representation of the double coverings of the direct sum of a dihedral group with $\ZZ_2$ in Appendix~\ref{app:doublecover} following Morris~\cite{Morris76}.

The main results of this article and its structure are now reviewed. Section~\ref{sec:DDDeq} reviews the background of Dunkl operators in general and in the dihedral cases. Section~\ref{sec:symalg} first presents the general results and definitions on the symmetry algebra, proves a useful proposition on the square of the central symmetry for any reflection group acting on a three-dimensional space (Proposition~\ref{prop:sqO123}), then constructs a new set of generators (Proposition~\ref{prop:commOdihedral}), continues with the construction of a pair of ladder operators (Proposition~\ref{prop:ladderdihedral}) and finishes with a discussion on the possible \modif{unitary structures}. Section~\ref{sec:Repr} states and proves the main results: a classification of the finite-dimensional irreducible and unitary representations (Theorems~\ref{thm:irrepsodd} and~\ref{thm:irrepseven} for the odd and even case respectively). The main ideas of the proofs are explained in Subsection~\ref{ssec:proofidea} and the details are given for the even case in Subsection~\ref{ssec:proofeven}. The proof of the odd case is very similar and has been omitted\footnote{Readers interested in the details can find them in the following version: \href{https://arxiv.org/abs/2010.03381v2}{arXiv:2010.03381v2}.}. Finally, Section~\ref{sec:Mono} studies the important example of the monogenic representation families (Propositions~\ref{prop:monoirrepodd}, and~\ref{prop:monoirrepeven}). For the convenience of the reader, we also included Appendix~\ref{app:doublecover} to recall results on the double coverings of reflections groups and write down the complete construction for $W=\ZZ_2\times D_{2m}$ of the irreducible finite-dimensional representations of the double coverings $\dcover{W}^{-}$ and $\dcover{W}^+$  (Theorem~\ref{thm:irrepsW}).

\section{The dihedral Dunkl--Dirac equation}\label{sec:DDDeq}

In this section, the required theory of Dunkl operators is recalled, both in the general case and in the specific dihedral case. The book of Dunkl and Xu~\cite{dunkl2014orthogonal} contains the material for Dunkl operators. The definitions of the generalized Dunkl--Laplace and Dunkl--Dirac operators can be found in~\cite{de_bie_algebra_2018}.

\subsection{The Dunkl--Laplace and Dunkl--Dirac operators}\label{ssec:DunklLaplDunklDirac}

Dunkl operators~\cite{dunkl_differential-difference_1989} are constructed by taking a reflection group $W\subset \mathcal{O}(N)$ (or a complex reflection group~\cite{dunkl_dunkl_2001}) and its associated  root system $R$ along with a $W$-invariant function $\kappa:R\to \CC$. Let the canonical bilinear form on $\RR^N$ be denoted by $\bili{-}{-}$. To any root $r\in R$, associate the reflection $\sigma_r\in W$ by the hyperplane orthogonal to $r$ 
\begin{equation}
	\sigma_r(x) = x - 2\frac{\bili{x}{r}}{\bili{r}{r}}r.
\end{equation}
The action of $W$ can be extended to the space of functions $f:\RR^N\to\CC$ by 
\begin{equation}
	wf(x) := f(w^{-1}x), \quad w\in W, \, x\in \RR^N.
\end{equation}
The $W$-invariance of $\kappa$ means it is constant on the $W$-orbits of $R$. (For more on reflection groups and root systems, see Humphreys's book~\cite{humphreys_reflection_1990}.)

Let the unit vector in position $j$ be denoted by $\xi_j$. The Dunkl operator linked to the coordinate $j$ is given then by
\begin{equation}\label{eq:defDun}
	\Dun{j}(f)(x):= \partial_{x_j}(f)(x) + \sum_{r\in R_+} \kappa(r)\bili{\xi_j}{r} \frac{f(x) -\sigma_r f(x)}{\bili{x}{r}} .
\end{equation}

This definition also enables to give a definition of $\Dun{v}$ for any vector $v\in\RR^N$. It is a well-known fact of Dunkl operators that $w\Dun{v}w^{-1} = \Dun{w\cdot v}$ for $v\in\RR^N$ and $w\in W$. Dunkl operators constitute a set of commuting differential operators~\cite{dunkl_differential-difference_1989}
\begin{equation}
	\comm{\Dun{i},\Dun{j}} = 0.
\end{equation}
Thus, it is justified to consider them as a generalisation of partial derivatives. 

The group $W$, the multiplication operators $x_1,\dots , x_N$ and the Dunkl derivatives $\Dun{1},\dots, \Dun{N}$ subject to the commutation relations, assuming the roots are normalized,
\begin{equation}\label{eq:commDx}
	\comm{\Dun{i},x_j} =  \delta_{ij} + \sum_{r\in R^+} 2\kappa(r) \bili{\xi_i}{r}\bili{\xi_j}{r} \sigma_{r}
\end{equation}
form an algebra that we denote $\alg{A}$. The algebra $\alg{A}$ is a faithful representation of a rational Cherednik algebra and it acts on polynomials. It is the appropriate formalism for defining operators analogous to their classical counterparts. For example, the Laplacian in the Dunkl setting is given  by
\begin{equation}
	\Lapl_{\kappa}:= \sum_{j=1}^N \Dun{j}^2.
\end{equation}
It is a homogeneous  endomorphism of degree $-2$ on the space of polynomials~\cite{dunkl_differential-difference_1989}. 

Define three elements of $\alg{A}$. Let $\gamma$ denote the sum of the values of $\kappa$ over the positive roots, $\sumx$ be the sum of the square of $x_i$ and $\Euler$ be the classical Euler operator
\begin{align}\label{eq:gamx2eul}
\gamma&:= \sum_{\alpha\in R_+} \kappa(\alpha), & \sumx&:=\sum_{j=1}^N x_i^2, & \Euler&:= \sum_{j=1}^N x_j \partial_{x_j}.
\end{align}
These operators yield an $\Lie{sl}_2$ realisation by the formulas extracted from~\cite{de_bie_algebra_2018}:
\begin{align}\label{eq:commsl2}
\comm{\Euler,\sumx} &= 2\sumx, & \comm{\Euler,\Lapl_{\kappa}} &= -2\Lapl_{\kappa}, & \comm{\sumx,\Lapl_{\kappa}} &= -4(\Euler + N/2+\gamma).
\end{align}
\modif{Note that in~\cite{de_bie_algebra_2018}, the symbol $\Euler$ was used to denote (when specialized to the Dunkl case)}
\begin{equation}
 \modif{\frac{1}{2}\sum_{j=1}^N \acomm{x_j,\Dun{j}} = \sum_{j=1}^N (x_j\Dun{j} + \comm{\Dun{j},x_j}/2) = \sum_{j=1}^N x_j\partial_{x_j} + N/2 + \gamma,}
\end{equation} 
\modif{where the last equality follows by mean of~\eqref{eq:defDun} and~\eqref{eq:commDx}. It includes the constants $N$ and $2\gamma$, hence the difference in the commutation relations~\eqref{eq:commsl2} and also~\eqref{eq:commddopxun}. Here, we follow the definition of~\cite{debie_total_2018}.}

Consider a Clifford algebra $\Clif(N)$ of rank $N$ with Clifford generators $e_1,\dots, e_N$ satisfying
\begin{align}
	e_j^2 &= 1, 	& 	\acomm{e_i,e_j}&= 0,\quad (i\neq j).
\end{align}
In this Clifford algebra setting, define the \emph{Dunkl--Dirac operator} and its dual symbol the \emph{vector variable} by
\begin{align}\label{eq:osp12DD}
	\DDop &:= \sum_{j=1}^N e_j\Dun{j}, & & \quad \xun:=\sum_{j=1}^N e_jx_j, &&\text{so that} & \DDop^2 &= \Lapl_{\kappa}, & \xun^2 &= \sumx.
\end{align}
The anti-commutation relation between $\DDop$ and $\xun$ is then given by
\begin{equation}\label{eq:commddopxun}
\acomm{\DDop,\xun} = 2(\Euler + N/2 + \gamma).
\end{equation}
Formally, $\DDop$ and $\xun$ lie in the tensor product $\Clif(N)\otimes \alg{A}$. We omit the tensor product symbol, trusting the different notations will be sufficient to permit the reader to avoid confusion.

The algebra $\Clif(N)\otimes \alg{A}$  contains a realisation of the Lie superalgebra $\osp(1|2)$ as the two following relations hold~\cite{orsted_howe_2009,de_bie_algebra_2018}:
\begin{align}
\comm{\acomm{\DDop,\xun},\DDop} &= -2\DDop, & \comm{\acomm{\DDop,\xun},\xun} &= 2\xun,
\end{align}
which \modif{is a} well-known presentation of the Lie superalgebra  $\osp(1|2)$~\cite{ganchev_lie_1980}.

Note that putting $e_j^2 =-1$ has an influence on the representation theory as it will change the double covering of the group $W$. We decided to keep the convention used in the study of the $S_3$ case~\cite{debie_total_2018} to help the reader
compare. The definitions of the symmetries and their relations are given for both options in~\cite{de_bie_algebra_2018} and the different outcomes of double coverings
are covered in Corollary~\ref{coro:doublecoverings}.

Before specialising to the three-dimensional case, one general result on some specific elements of $\alg{A}$ will be useful. Throughout the article, we make use of the following short-hand notations:
\begin{gather}
C_{ij} := \comm{\Dun{i},x_j} \stackrel{\eqref{eq:commDx}}{=} \comm{\Dun{j},x_{i}}, \qquad  C_{ij} = C_{ji},\\
L_{ij} := x_i\Dun{j} - x_j\Dun{i}, \qquad L_{ij} = -L_{ji}, \qquad L_{ii} = 0.
\end{gather}
The following theorem relates the $L_{ij}$'s and the $C_{ij}$'s for general root systems. In particular, equation~\eqref{eq:sumLLandLC} will be useful in the proof of Proposition~\ref{prop:sqO123}.
\begin{theorem}[{\cite[Theorem~2.5]{de_bie_algebra_2018} and \cite[Proposition~3.1]{feigin_dunkl_2015}}]
	 Let $i$, $j$, $k$, $l$ be four non-necessarily distinct integers in  $\{1,\dots, N\}$. The commutation relation between $L_{ij}$ and $L_{kl}$ is given by 
	\begin{align}
	\comm{L_{ij}, L_{kl}} &= L_{il}C_{jk} + L_{jk}C_{il} + L_{ki}C_{lj} + L_{lj}C_{ki}\\
		&= C_{jk}L_{il} + C_{il}L_{jk} + C_{lj}L_{ki} + C_{ki}L_{lj},
	\end{align}
	and the following identities hold
	\begin{gather}
	\acomm{L_{ij}, L_{kl}} + \acomm{L_{ki}, L_{jl}} + \acomm{L_{jk}, L_{il}} =0,\\
	\comm{L_{ij}, C_{kl}} + \comm{L_{ki}, C_{jl}} + \comm{L_{jk}, C_{il}} = 0,\\
	L_{ij}L_{kl} + L_{ki}L_{jl} + L_{jk}L_{il} = L_{ij}C_{kl} + L_{ki}C_{jl}+L_{jk}C_{il}.\label{eq:sumLLandLC}
	\end{gather}
\end{theorem}

\subsection{The dihedral Dunkl operators}

Let $I_2(m)$ denote the root system associated with the dihedral group $D_{2m}$ of order $2m$. For $m=1, 2, 3, 4, 6$, it is a crystallographic root system, respectively $A_1$, $A_1\oplus A_1$, $A_2$, $B_2$ and $G_2$. 

We will consider for the remainder of the article $m\geq 2$ and $N=3$ and put $W= \ZZ_2\times D_{2m}$ acting on $\RR^3$ with Coxeter presentation given by
\begin{equation}
W = \left\langle \sigma_0,\, \sigma_1,\, \sigma_m \mid \sigma_0^2 = \sigma_1^2 = \sigma_m^2 = (\sigma_0\sigma_1)^2 = (\sigma_0\sigma_m)^2 = (\sigma_1\sigma_m)^m =1\right\rangle,
\end{equation}
where $\sigma_0$ is the generator for $\ZZ_2$. We choose the standard root system of $A_1 \oplus I_2(m)$ as
\begin{align}
\alpha_0 &= (0,\,0,\,1) & \alpha_j &= (\sin(j\pi/m),\,-\cos(j\pi/m),\,0), \quad j=1,\dots, 2m,
\end{align}
with the set of positive roots  given by $R_+ = \left\{\alpha_0, \alpha_1, \dots, \alpha_m\right\}$. 

Before proceeding any further, we warn the reader that in previous works on $A_2 = I_2(3)$~\cite{debie_total_2018} and $G_2 = I_2(6)$~\cite{LROVarna19},  the root system used is the natural embedding of the roots in the $(1,1,1)$-hyperplane. In the $A_2$ case, the change of variables to $u$, $v$ and $w$ corresponds to $x_1=v$, $x_2=u$ and $w=x_3$ in this article.  Here we decided to follow the same convention for the dihedral groups as Dunkl~\cite{dunkl_differential-difference_1989} and Humphreys~\cite{humphreys_reflection_1990}. The associated reflections $\sigma_j$ are given in matrix form by
\begin{align}
\sigma_0 &= 
\begin{pmatrix}
1&0&0\\
0&1&0\\
0&0&-1
\end{pmatrix}, & 
\sigma_j &= 
\begin{pmatrix}
\cos (2j\pi/m) & \sin(2j\pi /m) & 0 \\
\sin(2j\pi/m) & -\cos(2j\pi/m) &0 \\
0&0&1
\end{pmatrix}.
\end{align}

The structure of dihedral groups depends on whether $m$ is even or odd. When it is even, the elements $\sigma_1$ and $\sigma_m$ are in two different conjugacy classes; when $m$ is odd, they are in the same. This has an impact on the double coverings (see Appendix~\ref{app:doublecover}) of the group and will impact slightly the representation theory.

With this in mind, a $W$-invariant function $\kappa$ is defined by at most three constants $\kappa_0 := \kappa(\alpha_0)$, $\kappa_1 := \kappa(\alpha_1)$ and $\kappa_m := \kappa(\alpha_m)$ linked to the $W$-orbits of $\alpha_0$, $\alpha_1$ and $\alpha_m$ respectively. Understand that $\kappa_1=\kappa_m$ when $m$ is odd. To this effect then, for positive $j$, $\kappa(\alpha_j) = \kappa_1$ when $m$ is odd; and $\kappa(\alpha_{2j}) = \kappa_m$, $\kappa(\alpha_{2j+1})=\kappa_1$ when $m$ is even. The Dunkl operators are then given, when $m$ is odd, by
\begin{equation}
\begin{aligned}
\Dun{1}(f(x)) &= \partial_{x_1}f(x) + \kappa_1 \sum_{j=1}^m  \frac{\sin(\tfrac{j\pi}{m})(f(x) - \sigma_j f(x))}{\sin (\tfrac{j\pi}{m}) x_1 - \cos(\tfrac{j\pi}{m}) x_2} ;\\
\Dun{2}(f(x)) &= \partial_{x_2} f(x) -\kappa_1 \sum_{j=1}^m \frac{\cos(\tfrac{j\pi}{m})(f(x) - \sigma_j f(x))}{\sin(\tfrac{j\pi}{m}) x_1 - \cos(\tfrac{j\pi}{m}) x_2} ;\\
\Dun{3}(f(x)) &= \partial_{x_3} f(x)  + \kappa_0 {f(x) - \sigma_0 f(x)\over x_3};
\end{aligned}
\end{equation}
and, when $m$ is even, by
\begin{align}
\Dun{1}(f(x)) &= \partial_{x_1}f(x) + \kappa_1\sum_{j=1\atop j \text{ odd}}^{m-1}  \frac{\sin(\tfrac{j\pi}{m})(f(x) - \sigma_{j} f(x))}{\sin (\tfrac{j\pi}{m}) x_1 - \cos(\tfrac{j\pi}{m}) x_2}  \nonumber
+ \kappa_m \sum_{j=1\atop j\text{ even}}^{m}\frac{\sin(\tfrac{j\pi}{m})(f(x) - \sigma_{j} f(x))}{\sin (\tfrac{j\pi}{m}) x_1 - \cos(\tfrac{j\pi}{m}) x_2} ; \nonumber\\
\Dun{2}(f(x)) &= \partial_{x_2} f(x) -  \kappa_1\sum_{j=1\atop j \text{ odd}}^{m-1}  \frac{\cos(\tfrac{j\pi}{m})(f(x) - \sigma_{j} f(x))}{\sin (\tfrac{j\pi}{m}) x_1 - \cos(\tfrac{j\pi}{m}) x_2} 
- \kappa_m \sum_{j=1\atop j\text{ even}}^{m}\frac{\cos(\tfrac{j\pi}{m})(f(x) - \sigma_{j} f(x))} {\sin (\tfrac{j\pi}{m}) x_1 - \cos(j\pi/m) x_2} ;\\
\Dun{3}(f(x)) &= \partial_{x_3} f(x)  + \kappa_0 \frac{f(x) - \sigma_0 f(x)}{x_3}.\nonumber
\end{align}

For this reflection group $W= \ZZ_2\times D_{2m}$, equation~\eqref{eq:commddopxun} becomes
 \begin{equation}
 \acomm{\DDop,\xun} = 2 \left( \Euler + \frac32 +   \frac{m}{2}(\kappa_1 + \kappa_m) + \kappa_0\right).
 \end{equation}

\section{The symmetry algebra of the Dunkl--Dirac operator}\label{sec:symalg}

In this section, we define the algebra we study by giving a generating set of elements supercommuting with the Dunkl--Dirac operator $\DDop$ and the vector variable $\xun$. The definition is not restricted to the dihedral case and thus we take the opportunity to prove a result, Proposition~\ref{prop:sqO123}, that holds for any reflection group $W$ acting on $\RR^3$. We then return to the dihedral case and prove the main result of the section, Proposition~\ref{prop:ladderdihedral}, that exhibits a pair of ladder operators and the factorisations of their products. The section ends with a small discussion on the unitary structure considered.

\subsection{General symmetry algebra for 3D space}
We study elements of the algebra $\Clif(3)\otimes\alg{A}$ with general reflection group $W$ and $W$-invariant function $\kappa$. We begin by presenting elements, called \emph{symmetries}, that supercommute with the Dunkl--Dirac operator and the vector variable. They were defined, and their relations  studied, in~\cite{de_bie_algebra_2018}.

 The (positive) double covering $\dcover{W}$ of the reflection group $W$ is the first instance of such symmetries since its elements supercommute with $\DDop$ and $\xun$~\cite{de_bie_algebra_2018}. By viewing $W$ as a subgroup of the orthogonal group $\mathcal{O}(3)$, we obtain $\dcover{W}$ as the pullback of the projection of the double covering $\Pin_+(3)$ onto $\mathcal{O}(3)$. In the realisation $\dcover{W}\subset \Clif(3)\otimes \mathcal{A}$, its generators are obtained as
 \begin{equation}
 \dsig{r}:= \sum_{j=1}^3(\bili{r}{\xi_j}e_j) \otimes \sigma_r.
 \end{equation}
 	Alternatively, a definition in terms of abstract generators and relations is available in Appendix~\ref{app:doublecover}. The commuting element $z$ in the abstract presentation is realised as $-1$ in $\Clif(3)\otimes A$, thus the group algebra in our realisation is in fact a quotient of the abstract group algebra by $z=-1$.

We continue with three types of symmetries linked to polynomial expressions in Clifford variables. The \emph{one-index symmetries} have the following general expression:
 \begin{equation}\label{eq:Oigen}
 	\sym{i} := {1\over 2} (\comm{\DDop,x_i} - e_i) = {1\over 2}(\comm{\Dun{i},\xun} - e_i) = {1\over 2} \left(\sum_{k=1}^3e_kC_{ki}-e_i\right).
 \end{equation}

The one-index symmetries are included in the group algebra $\CC\dcover{W}$. They are however useful in order to simplify future expressions. If the root system is normalized, they can be rewritten in terms of the elements $\dsig{j}$ as~\cite[Ex. 4.2]{de_bie_algebra_2018}
\begin{equation}\label{eq:Oibydsig}
\sym{i} = \sum_{j=0}^m \kappa(\alpha_j) \bili{\xi_i}{\alpha_j}\dsig{j}.
\end{equation}

The following lemma shows how the Clifford elements interact with the one-index symmetries.
\begin{lemma}[{\cite[Lem.~3.10]{de_bie_algebra_2018}}]\label{lem:sym310}
	For any two indices $i$, $j$ the following relations hold
	\begin{equation}
	\acomm{e_i,\sym{j}} = \acomm{e_j,\sym{i}}.
	\end{equation}
\end{lemma} 
The \emph{two-index symmetries} are defined below, with a second expression following Lemma~\ref{lem:sym310}
\begin{align}
\sym{ij}&:= L_{ij} + \tfrac{1}{2} e_ie_j + \sym{i}e_j - \sym{j}e_i, \label{eq:Oij1}\\
&= L_{ij} +\tfrac{1}{2} e_ie_j + e_i\sym{j} - e_j\sym{i}. \label{eq:Oij2}
\end{align}
Finally, the \emph{three-index symmetry} has also two equivalent expressions given by 
\begin{equation}\label{eq:O123ori}
\sym{123} := \tfrac{1}{2}(\comm{\DDop,\xun}-1)e_1e_2e_3 = -\tfrac12e_1e_2e_3(\comm{\DDop,\xun} - 1),
\end{equation}
but the following expansions will be more useful to work with
\begin{align}
\sym{123}&= -\tfrac{1}{2} e_1e_2e_3 - \sym{1}e_2e_3 - \sym{2}e_3e_1 - \sym{3}e_1e_2 + \sym{12}e_3 + \sym{31}e_2 + \sym{23}e_1, \label{eq:O123v1}\\
&= -\tfrac{1}{2}e_1e_2e_3 - e_2e_3 \sym{1} - e_3e_1 \sym{2} - e_1e_2 \sym{3} + e_3\sym{12} + e_2 \sym{31} + e_1\sym{23}.\label{eq:O123v2}
\end{align}
In \eqref{eq:O123ori}, the factor $\tfrac{1}{2}(\comm{\DDop,\xun} -1)$ is the Scasimir of $\osp(1|2)$~\cite{frappat_dictionary_2000}.

A word of warning: albeit the last equations make it look so, the Clifford elements do not in general commute or anticommute with the symmetries; only certain combinations of Clifford elements can commute following Lemma~\ref{lem:sym310}. However, the $\tot{ij}$'s and $C_{ij}$'s, being purely elements of $\alg{A}$, commute freely with Clifford elements.

We are ready to define the algebra studied in this paper. It is given as a subalgebra of $\Clif(3)\otimes \alg{A}$ generated by the elements presented above. The name of the algebra is derived from the fact that all its elements supercommute with the $\osp(1|2)$ realisation. 
\begin{definition}\label{def:symalg}
	The \emph{symmetry algebra} $\Asym$ is \modif{the associative subalgebra} of $\Clif(3)\otimes \alg{A}$ generated by the symmetries $\sym{12}$, $\sym{31}$, $\sym{23}$, $\sym{123}$ and the group $\dcover{W}$.
\end{definition}

We will give the commutation relations for all the generators only for the dihedral cases because a different set of generators will be better suited for the following results. The new set is given by equations~\eqref{eq:defTpm} and~\eqref{eq:defO0pm}, and the relations are found in Lemma~\ref{lem:commWO0} and Proposition~\ref{prop:commOdihedral}. In the general case, we only need to say that the three-index symmetry $\sym{123}$ commutes with every element of the symmetry algebra and state the commutation rules of the two-index symmetries. The other relations needed in the proof of Proposition~\ref{prop:sqO123} are retrieved implicitly from the definitions of the elements.

The following proposition shows clearly that the two-index symmetries relations are an extension of the Lie algebra $\Lie{so}(3)$ commutation rules \modif{as taking $\kappa$ to be the zero map would send the one-index symmetries to $0$.}

\begin{proposition}[\cite{de_bie_algebra_2018}]\label{prop:comm3Dsym}
	The two-index symmetries commutation rules are given by
	\begin{equation}\label{eq:commutationrules3d}
	\begin{aligned}
	\comm{\sym{12},\sym{31}} &= \sym{23} + \acomm{\sym{123},\sym{1}} + \comm{\sym{2},\sym{3}};\\
	\comm{\sym{23},\sym{12}} &= \sym{31} + \acomm{\sym{123},\sym{2}} + \comm{\sym{3},\sym{1}};\\
	\comm{\sym{31},\sym{23}} &= \sym{12} + \acomm{\sym{123},\sym{3}} + \comm{\sym{1},\sym{2}}.
	\end{aligned}
	\end{equation}	
\end{proposition}

It will be useful later on to have an expression for the square of $\sym{123}$. \modif{Indeed, as $\sym{123}$ is the product of the Scasimir of $\osp(1|2)$ with the pseudo-scalar $e_1e_2e_3$ (see equation~\eqref{eq:O123ori}), its square is the Casimir of $\osp(1|2)$. Proposition~\ref{prop:sqO123} expresses $\sym{123}^2$ as a sum of the squares of the other symmetries (considering a trivial symmetry $\sym{\emptyset} = i/2$). This sum is thus central and furthermore, when $\kappa$ is set to $0$, it reduces to the Casimir of the undeformed $\Lie{so}(3)$ algebra. } We emphasize that this result does not assume anything on $W$ outside it acting on a three-dimensional space.

\begin{proposition}\label{prop:sqO123}
	For any reflection group acting on $\RR^3$, the three-index symmetry $\sym{123}$ squares to
	\begin{equation}\label{eq:sqO123}
	\sym{123}^2 = -{1\over 4} + \sym{1}^2 + \sym{2}^2 + \sym{3}^2 + \sym{12}^2 + \sym{31}^2 + \sym{23}^2.
	\end{equation}
\end{proposition}

\begin{proof}
	The first step in the proof is to use the two expressions~\eqref{eq:O123v1} and~\eqref{eq:O123v2} for $\sym{123}$ to put the Clifford elements in the middle:
	\begin{align*}
	\sym{123}^2 &= \left(\tfrac{1}{2} e_1e_2e_3 - \sym{1}e_2e_3 - \sym{2}e_3e_1 - \sym{3}e_2e_3 + \sym{12}e_3 + \sym{31}e_2 + \sym{23}e_1\right)\\
	&\quad \times \left(\tfrac{1}{2}e_1e_2e_3 - e_2e_3 \sym{1} - e_3e_1 \sym{2} - e_1e_2 \sym{3} + e_3\sym{12} + e_2 \sym{31} + e_1\sym{23}\right).
	\end{align*}
	In working out the 49 terms of this product, separate the 7 ``diagonal terms'' and the 42 ``cross terms''. Simplifying with the  Clifford anticommutation relations gives
	\begin{align}
		\sym{123}^2 &= -{1\over 4} - \sym{1}^2 - \sym{2}^2 - \sym{3}^2 + \sym{12}^2 + \sym{31}^2 + \sym{23}^2 + Q,
	\end{align}
    with $Q$ consisting of the 42 cross terms, all of them with symmetries shouldering Clifford elements. The proof is completed once it is shown that $Q$ reduces to $2(\sym{1}^2+\sym{2}^2+\sym{3}^2)$.
   
    For this purpose, replace the two-index symmetries at the left of the central Clifford elements by their definition with Clifford elements on the right, equation~\eqref{eq:Oij1}; and replace the two-index symmetries at the right of the Clifford elements by their definition with Clifford elements on the left, equation~\eqref{eq:Oij2}. After simplifications, this will give
\begin{align}
	Q &=  2( \sym{1}^2 + \sym{2}^2 + \sym{3}^2)\\
	&\quad+ {1\over 2}\left( \begin{array}{l}
	\tot{12}e_1e_2 (1-e_3e_1\tot{31} - e_2e_3\tot{23}  + 2e_3\sym{3})\\
	+\ \tot{31}e_3e_1 (1-e_1e_2\tot{12}-e_2e_3\tot{23}  + 2e_2\sym{2})\\
	+\ \tot{23}e_2e_3(1-e_1e_2\tot{12} - e_3e_1\tot{31}  + 2e_1\sym{1})
	\end{array}\right)\label{eq:sqO123B2}\\
	&\quad+ \frac12 \left( \begin{array}{l}
	(1-\tot{31}e_3e_1 - \tot{23}e_2e_3 + 2\sym{3}e_3)e_1e_2\tot{12}\\
	+\ (1-\tot{12}e_1e_2 - \tot{23}e_2e_3 + 2\sym{2}e_2)e_3e_1\tot{31}\\
	+\ (1-\tot{12}e_1e_2 - \tot{31} e_3e_1 + 2\sym{1}e_1)e_2e_3\tot{23}
	\end{array}\right). \label{eq:sqO123B3}
	\end{align}

 	The final step is to show that each of the two last components~\eqref{eq:sqO123B2} and~\eqref{eq:sqO123B3} are $0$. In each of them, replace the one-index symmetries by their expression~\eqref{eq:Oigen} in terms of Clifford elements and $C_{ij}$. As the $\tot{ij}$'s and $C_{ij}$'s commute with Clifford elements, factor out the Clifford elements. The coefficient of the Clifford variables will be sums of $\tot{ij}$ and $C_{ij}$. They will cancel out by equation~\eqref{eq:sumLLandLC}.
    This gives the appropriate expression for $Q$, proving the lemma.
\end{proof}

\subsection{Dihedral 3D symmetry algebra and ladder operators}

Let $m\geq 2$. We now specify to $W = \ZZ_2\times D_{2m}$. We denote the symmetry algebra linked to this $W$ by $\SymAlg$ as opposed to the general symmetry algebra $\Asym$. We begin by giving the explicit expressions of the elements of $\dcover{W}$: 
\begin{align}
\dsig{0} &= e_3\sigma_0, & \dsig{j} = (\sin(j\pi/m)e_1 - \cos(j\pi/m)e_2 )\sigma_j.
\end{align}
They generate the group $\dcover{W}\subset \Clif(3)\otimes \CC[W]$ with presentation given by
\begin{equation}
	\dcover{W} = \left\langle \dsig{0},\, \dsig{1},\, \dsig{m}\ \middle| \ \dsig{0}^2 = \dsig{1}^2=\dsig{m}^2 = 1, \, (\dsig{0}\dsig{1})^2=(\dsig{0}\dsig{m})^2 = -1, \ (\dsig{1}\dsig{m})^m = (-1)^{m+1}\right\rangle.
\end{equation}
Note that because the group is realised in $\Clif(3)\otimes \CC[W]$, there is no need to add $-1$ as a generator. It is the positive double covering of $W$ (or the negative double covering if the Clifford elements square to $-1$ instead of $+1$, see Corollary~\ref{coro:doublecoverings}). Appendix~\ref{app:doublecover} studies the double coverings abstractly. A few things can immediately be said of $\dcover{W}$ nonetheless. The group depends on the parity of $m$, and the generator $\dsig{0}$ of the $\ZZ_2$ part \emph{anti}-commutes with $\dsig{1}$ and $\dsig{m}$ whereas $\sigma_0$ commutes with $\sigma_1$ and $\sigma_m$; so $\dcover{W} \not\simeq W$ in general. This is a difference with the case $W=S_3$ as $S_3$ does not have a non-trivial positive double covering~\cite{Schur1907} (it can also readily be  seen from Corollary~\ref{coro:doublecoverings}), however $S_3\times \ZZ_2$ does. 

The remainder of the section is dedicated to finding ladder operators, which will be crucial for the construction of representations of $\SymAlg$. Inspired by the construction of ladder operators in the $\Lie{so}(3)$ case, we define
\begin{align}\label{eq:defO0pm}
\sym{0} &:= -i \sym{12}, & \sym{+} &:= i\sym{31} +  \sym{23}, & \sym{-} &:= i\sym{31} - \sym{23}.
\end{align}
The algebra $\SymAlg$ is also generated by $O_0$, $\sym+$, $\sym-$, $\sym{123}$ and $\dcover{W}$.

Define the following combinations of elements of the group algebra of $\dcover{W}$
\begin{align}\label{eq:defTpm}
T_0 := i\kappa_0 \dsig{0}, & &T_+ &:= -i\sum_{j=1}^m \kappa(\alpha_j) e^{j\pi i/m} \dsig{j}, & T_- := i\sum_{j=1}^m\kappa(\alpha_j) e^{-j\pi  i/m} \dsig{j}.
\end{align}
They are linked with the one-index symmetries by equation~\eqref{eq:Oibydsig}:
\begin{align}
T_0&= i\sym{3},& T_+ &= \sym1 + i \sym2, & T_- &= \sym1 - i\sym{2}, & \comm{T_+,T_-} &= -2i\comm{\sym{1},\sym{2}}.
\end{align}

Furthermore, they can be expressed in another form according to the parity of $m$. Put $\zeta := e^{\pi i /m}$. When $m$ is \textbf{odd}, then all the $\kappa(\alpha_j)$ are the same and so
\begin{equation}
T_+ = -i \kappa_1 \sum_{j=1}^m \zeta^j \dsig{j}, \qquad T_- =  i \kappa_1 \sum_{j=1}^m \zeta^{-j} \dsig{j},
\end{equation}
and when $m=2p$ is \textbf{even} then it is expressed as
\begin{equation}\label{eq:Tpmeven}
T_{+} = - i(\kappa_1 T_{+}^1 + \kappa_m T_{+}^2), \qquad T_{-} = + i(\kappa_1 T_{-}^1 + \kappa_m T_{-}^2),
\end{equation}
with $T_{\pm}^{1}$, and $T_{\pm}^2$ the sums over odd and even indices:
\begin{equation}
\begin{aligned}
T_+^1 &= \sum_{j=1}^p \zeta^{2j-1} \dsig{2j-1}, & T_+^2 &= \sum_{j=1}^p\zeta^{2j} \dsig{2j},&
T_-^1 &= \sum_{j=1}^p \zeta^{1-2j} \dsig{2j-1}, & T_-^2 &= \sum_{j=1}^p\zeta^{-2j} \dsig{2j}.
\end{aligned}
\end{equation}

The next lemma gives useful commutation properties between the new generators.
\begin{lemma}\label{lem:commWO0}
	The element $\sym{0}$ has the following commutation relations with elements of $\CC[\dcover{W}]$: 
	\begin{equation}\label{eq:commWO0O0}
	\comm{\sym{0},T_0} =\acomm{\sym{0},T_+} = \acomm{\sym{0},T_-} =  0.
	\end{equation}
	Furthermore, $T_0$, $T_+$ and $T_-$ interact with $\sym{-}$ and $\sym{+}$ as follows:
	\begin{align}
	T_0\sym{+} &= - \sym{+}T_0, & T_0\sym{-} &= -\sym{-}T_0,\label{eq:T0andO}\\
	T_0T_+ &= -T_+T_0, & T_0 T_- &= -T_-T_0,\label{eq:T0andT}\\
	T_+\sym{-} &= -\sym{+}T_-, & T_-\sym{+} &= -\sym{-}T_+.\label{eq:TandO}
	\end{align}
\end{lemma}
\begin{proof}
	Equations~\eqref{eq:commWO0O0} are directly equivalent to $\comm{\sym{12}, \dsig0} =0$ and $\acomm{\sym{12},\dsig{j}} = 0$ for $1\leq j \leq m$. For the first, $\comm{\sym{0},T_0} =0$, it follows from the action of $\sigma_0$ that $\dcover{\sigma}_0$, and so $T_0$, will commute with $\sym{12}$. For the two others, $\acomm{\sym{0},T_+}=0 = \acomm{\sym{0},T_-}$, expand $\dsig{j}$ and $\sym{12}$ by their definition to obtain a product of Clifford elements with one-index symmetries and $\tot{12}$. We give the computations for $\sigma_{j}\tot{12}$,
	\begin{align*}
		\sigma_j\tot{12} &= \sigma_j(x_1\Dun{2} - x_2\Dun{1})  \\
		&= \left(\cos(2\pi j/m) x_1 + \sin(2\pi j/m)x_2\right)\left(\sin(2\pi j/m)\Dun{1} - \cos(2\pi j/m)\Dun{2}\right)\sigma_j\\
		&\quad - \left(\sin(2\pi j/m) x_1 - \cos(2\pi j/m)x_2\right)\left(\cos(2\pi j/m)\Dun{1} + \sin(2\pi j/m)\Dun{2}\right)\sigma_j\\
		&= (\sin^2(2\pi j/m) + \cos^2(2\pi j/m)) \tot{21}\sigma_j = -\tot{12}\sigma_j.
	\end{align*}
    The computations on the one-index symmetries follow from their definition~\eqref{eq:Oigen} and the Clifford part is direct, so $\acomm{\sym{12},\dsig{j}} =0$.
	
	For the next expressions~\eqref{eq:T0andO},~\eqref{eq:T0andT} and~\eqref{eq:TandO}, remark that $\sym{3}$ leaves $\sym{12}$ invariant and  sends $\sym{31}$ and $\sym{23}$ to $-\sym{31}$ and $-\sym{23}$ respectively, so $\sym{3}\sym{+} = -\sym{+}\sym{3}$ and $\sym{3}\sym{-} = -\sym{-}\sym{3}$. This proves equations~\eqref{eq:T0andO}. Additionally, $\dsig{0}$ and $\dsig{j}$ anticommute because $e_3$ anticommutes with $e_1$ and $e_2$, and because $\sigma_{0}$ commutes with $\sigma_{j}$. Therefore, $\acomm{\sym{3},\sym{1}} = 0 = \acomm{\sym{3},\sym{2}}$, and so, $\acomm{\sym{3},T_{\pm}} = 0$; proving equations~\eqref{eq:T0andT}.
	
	Finally, equations~\eqref{eq:TandO} are proven from the expression~\eqref{eq:defTpm}. By direct computations, we have that	$\dsig{j}\sym{\pm} = e^{\pm 2j\pi i/m}\sym{\mp}\dsig{j}$ and therefore
	\begin{align*}
	T_+\sym{-} &= -i\sum_{j=1}^m \kappa(\alpha_j)e^{\tfrac{j\pi i}{m}}\dsig{j} \sym{-} 
	= -i\sum_{j=1}^m \kappa(\alpha_j) e^{\tfrac{j\pi i}{m}}e^{\tfrac{-2j\pi i}{m} }\sym{+}\dsig{j}
	= -\sym{+}\left(i\sum_{j=1}^m \kappa(\alpha_j) e^{\tfrac{-j\pi i}{m}}\dsig{j}\right) 
	= -\sym{+}T_-,
	\intertext{and}
	T_-\sym{+} &= i\sum_{j=1}^m \kappa(\alpha_j)e^{\tfrac{-j\pi i}{m}}\dsig{j} \sym{+}
	= i\sum_{j=1}^m \kappa(\alpha_j) e^{\tfrac{-j\pi i}{m}}e^{\tfrac{2j\pi i}{m} }\sym{-}\dsig{j}
	= -\sym{-}T_+.
	\end{align*}
	This concludes the proof.
\end{proof}

With the new set of generators, Proposition~\ref{prop:comm3Dsym} translates to the following. 
\begin{proposition}\label{prop:commOdihedral}
	The linear combinations $\sym{0}$, $\sym{+}$ and $\sym{-}$ satisfy
	\begin{equation} \label{eq:commrelO0OpOm}
	\begin{aligned}
	\comm{\sym{0},\sym{+}} & = + \sym{+} + \acomm{\sym{123},T_+} + \comm{T_0,T_+};\\
	\comm{\sym{0},\sym{-}} &= - \sym{-} + \acomm{\sym{123},T_-} - \comm{T_0,T_-};\\
	\comm{\sym{+},\sym{-}} &= 2\sym{0} - 2\acomm{\sym{123},T_0} + \comm{T_+,T_-}.
	\end{aligned}
	\end{equation}
\end{proposition}
\begin{proof}
	It follows from the commutation rules \eqref{eq:commutationrules3d}  and the definitions of $T_0$, $T_+$ and $T_-$ that:
	\begin{align*}
	\comm{\sym{0},\sym{+}} &= \comm{-i\sym{12},i\sym{31} + \sym{23}}\\
	&= \comm{\sym{12},\sym{31}} - i\comm{\sym{12},\sym{23}}\\
	&= \sym{23} + \acomm{\sym{123},\sym{1}} + \comm{\sym{2}, \sym{3}} + i(\sym{31} + \acomm{\sym{123},\sym{2}} + \comm{\sym{3},\sym{1}})\\
	&= \sym{+}  + \acomm{\sym{123},T_+} + i\comm{\sym{3},T_+},	
	\end{align*}
	and similarly for the two other equations.
\end{proof}

As a corollary of Proposition~\ref{prop:sqO123}, the products $\sym{+}\sym{-}$ and $\sym{-}\sym{+}$  have new expressions.
\begin{corollary}\label{coro:sqO123}
	The products $\sym{+}\sym{-}$ and $\sym{-}\sym{+}$ can be expressed as	
	\begin{align}
	\sym{+}\sym{-} &= T_+T_- - (\sym{0} - 1/2)^2 - (\sym{123} + T_0)^2, \\
	\sym{-}\sym{+} &= T_-T_+ - (\sym{0} + 1/2)^2 - (\sym{123} - T_0)^2.
	\end{align}
\end{corollary}
\begin{proof}
	The formulas are obtained from changing the variables in the expression of $\sym{123}^2$ of Proposition~\ref{prop:sqO123}.  First note that $\sym{0}^2 = -\sym{12}^2$. Then compute
	\begin{align*}
	\sym{+}\sym{-} &= (i\sym{31} + \sym{23})(i\sym{31} - \sym{23})
	= -\sym{31}^2 - \sym{23}^2 - i \comm{\sym{31} , \sym{23}},
	\intertext{and thus}
	\sym{31}^2 + \sym{23}^2 &= -\sym{+}\sym{-} - i\comm{\sym{31},\sym{23}}\\
	&= -\sym{+}\sym{-}  -i \sym{12} - 2i\sym{123}\sym{3} - i\comm{\sym{1},\sym{2}}\\
	&=- \sym{+}\sym{-} + \sym{0} - 2i\sym{123}\sym{3} + \tfrac{1}{2}\comm{T_+,T_-}.
	\end{align*}
	Follow up with 
	\begin{align*}
		T_+T_- &=(\sym{1} + i\sym{2})(\sym{1}-i\sym{2}) = \sym{1}^2 + \sym{2}^2 - i\comm{\sym{1},\sym{2}},
		\intertext{hence,}
		\sym{1}^2 + \sym{2}^2 &= T_+T_- - \tfrac{1}{2}\comm{T_+,T_-}.
	\end{align*}
	Using the expression for $\comm{\sym{+},\sym{-}}$ to invert $\sym{+}\sym{-}$ and $T_+T_-$, we have the two equalities 
	\begin{align*}
		\sym{123}^2 &= -1/4 + T_+T_- - \sym{+}\sym{-} + \sym{3}^2 - \sym{0}^2 + \sym{0} - 2i\sym{123}\sym{3},\\
		&= -1/4 + T_-T_+ - \sym{-}\sym{+} +\sym{3}^2 - \sym{0}^2 - \sym{0} + 2i\sym{123}\sym{3},
	\end{align*}
	replacing $i\sym{3} = T_0$ and factorising finish the proof.
\end{proof}

The next proposition introduces the two ladder operators and their factorisations.
\begin{proposition}\label{prop:ladderdihedral}
	Consider $\Lad{+} = \tfrac{1}{2}\acomm{\sym{0},\sym{+}}$ and $\Lad{-} =  \tfrac{1}{2}\acomm{\sym{0},\sym{-}}$. They form a pair of ladder operators: 
	\begin{equation}\label{eq:laddihedral}
	\comm{\sym{0},\Lad+} = \Lad+, \quad \comm{\sym{0},\Lad-} = -\Lad-.
	\end{equation}
	Furthermore, they have the following factorisations:
	\begin{align}
	\Lad{+}\Lad{-} &= -\left((\sym0 - 1/2)^2 + (\sym{123}+T_0)^2\right)\left((\sym0 -1/2)^2 - T_+T_-\right), \label{eq:factoLpLm}\\
	\Lad{-}\Lad{+} &= -\left((\sym0 + 1/2)^2 + (\sym{123}-T_0)^2\right)\left((\sym0 + 1/2)^2 - T_-T_+\right). \label{eq:factoLmLp}
	\end{align}
\end{proposition}
\begin{proof}
	Start by expanding equation~\eqref{eq:laddihedral} by the definition of $\Lad{\pm}$:
	\begin{align*}
	2\comm{\sym{0},\Lad{\pm}} &= \comm{\sym{0},\acomm{\sym{0},\sym{\pm}}} = \acomm{\sym{0},\comm{\sym{0},\sym{\pm}}}\\
	&= \acomm{\sym{0},\pm\sym{\pm}} + \acomm{\sym{0}, \acomm{\sym{123},T_{\pm}}} + i\acomm{\sym{0},\comm{\sym{3},T_{\pm}}},
	\intertext{and as $\sym{123}$ is central, and $T_-$ and $T_+$ anticommute with $\sym{0}$ by Lemma~\ref{lem:commWO0}, it reduces to}
	\comm{\sym{0},\Lad{\pm}}&= \pm \Lad{\pm}.
	\end{align*}
	Before getting to the factorisations, the following claim
	\begin{equation}\label{eq:OpmOmpO0comm}
	\comm{\sym{+}\sym{-},\sym{0}} = 0 = \comm{\sym{-}\sym{+},\sym{0}},
	\end{equation}
	will be proven for the first case; the second being similar. Use Proposition~\ref{prop:commOdihedral} to replace the commutators of $\sym{+}$, $\sym{-}$ and $\sym{0}$, and use Lemma~\ref{lem:commWO0} to send $\sym{-}$ in front 
	\begin{align*}
	\comm{\sym{+}\sym{-},\sym{0}} &= \sym{+}\comm{\sym{-},\sym{0}} + \comm{\sym{+},\sym{0}}\sym{-} \\
	&= \sym{+}(\sym{-} - 2\sym{123}T_- + 2i\sym{3}T_-) + (-\sym{+} - 2i \sym{123}T_+ -2 i\sym{3}T_+)\sym{-}	=0.
	\end{align*}
	
	In order to give the factorisation, replace the ladder operators by their definitions and use the commutation relations of Proposition~\ref{prop:commOdihedral} and Lemma~\ref{lem:commWO0} to reach:
	\begin{align}
	\Lad{+}\Lad{-} &= (\sym{123} + T_0)^2T_+T_- + (\sym{0} - 1/2)^2\sym{+}\sym{-},\label{eq:LpLm1}\\
	\Lad{-}\Lad{+} &= (\sym{123} - T_0)^2T_-T_+ + (\sym{0} + 1/2)^2\sym{-}\sym{+}. \label{eq:LmLp1}
	\end{align}
	Since the actual computation is rather tricky, we will show the details to obtain equation~\eqref{eq:LpLm1} and trust the reader to do 
	the second. For clarity, we will add a factor $4$ to remove the fractions. Start by the definition of the ladder operators
	\begin{align*}
	4\Lad{+}\Lad{-} &= (\sym{0}\sym{+} + \sym{+}\sym{0})(\sym{0}\sym{-} + \sym{-}\sym{0})\\
	&=\sym{0}\underline{\sym{+}\sym{0}}\sym{-} + \sym{0}\underline{\sym{+}\sym{-}\sym{0}} + \underline{\sym{+}\sym{0}}\sym{0}\sym{-} + \underline{\sym{+}\sym{0}}\sym{-}\sym{0}.
	\end{align*}
	Apply to the underlined terms the commutation relations of Proposition~\ref{prop:commOdihedral} pertaining to $\comm{\sym{0},\sym{+}}$  and equation~\eqref{eq:OpmOmpO0comm}  to obtain 
	\begin{align*}
	4\Lad{+}\Lad{-}&= \sym{0}(\sym{0}\sym{+} - \sym{+} - 2\sym{123}T_+ - 2 i \sym{3}T_+)\sym{-} + \sym{0}^2\sym{+}\sym{-} \\
	&\quad + (\sym{0}\sym{+}-\sym{+}-2\sym{123}T_+ - 2i\sym{3}T_+)\sym{0}\sym{-}
	+ (\sym{0}\sym{+} - \sym{+} - 2\sym{123}T_+ - 2i\sym{3}T_+)\sym{-}\sym{0}
	\\
	&= 2\sym{0}^2\sym{+}\sym{-} - \sym{+}\sym{-} - 2\sym{123}T_+\sym{-} - 2i\sym{3}T_+\sym{-}\\
	&\quad + \sym{0}\underline{\sym{+}\sym{0}}\sym{-} - \underline{\sym{+}\sym{0}}\sym{-} - 2\sym{123}T_+\sym{0}\sym{-} - 2i\sym{3}T_+\sym{0}\sym{-}\\
	&\quad + \sym{0}\underline{\sym{+}\sym{-}\sym{0}} - \underline{\sym{+}\sym{-}\sym{0}} - 2\sym{123}T_+\underline{\sym{-}\sym{+}} - 2i\sym{3}T_+\underline{\sym{-}\sym{0}}.
	\end{align*}
	Use again Proposition~\ref{prop:commOdihedral} and equation~\eqref{eq:OpmOmpO0comm} on the underlined terms to get
	\begin{align*}
	4\Lad{+}\Lad{-}&= 4\sym{0}^2\sym{+}\sym{-} - 4\sym{0}\sym{+}\sym{-} + \sym{+}\sym{-}
	- 2\sym{123}T_+ (\sym{0}\sym{-} + \underline{\sym{-}\sym{0}} - \sym{-}) - 4\sym{0}\sym{123}T_+\sym{-}\\
	&\quad - 2i\sym{3}T_+(\sym{0}\sym{-} + \underline{\sym{-}\sym{0}} - \sym{-}) - 4i\sym{0}\sym{3}T_+\sym{-}.
	\end{align*}
	At this point, replace the last instances of $\sym{-}\sym{0}$ with Proposition~\ref{prop:commOdihedral} to obtain
	\begin{align*}
	4\Lad{+}\Lad{-} &= 4\sym{0}^2\sym{+}\sym{-} - 4\sym{0}\sym{+}\sym{-} + \sym{+}\sym{-}	- 2\sym{123}T_+ (2\underline{\sym{0}}\sym{-} -2\sym{123}T_- + 2i\underline{\sym{3}}T_- )\\
	&\quad - 2i\sym{3}T_+(2\underline{\sym{0}}\sym{-} -2\sym{123}T_- + 2i\underline{\sym{3}}T_- )	- 4\sym{0}\sym{123}T_+\sym{-} - 4i\sym{0}\sym{3}T_+\sym{-}.
	\end{align*}
Make use of Lemma~\ref{lem:commWO0} to then send the underlined $\sym{0}$ and $\sym{3}$ in front, which will give the expanded equation~\eqref{eq:LpLm1}
	\begin{align*}
	4\Lad{+}\Lad{-}	&= 4\sym{0}^2\sym{+}\sym{-} - 4\sym{0}\sym{+}\sym{-} + \sym{+}\sym{-} + 4\sym{123}^2T_+T_- + 8i\sym{123}\sym{3}T_+T_- - 4\sym{3}^2T_+T_-.
	\end{align*}
	
	With this, we proved the first factorisation~\eqref{eq:LpLm1}. To conclude, employ Corollary~\ref{coro:sqO123} to replace $\sym{+}\sym{-}$ in equation~\eqref{eq:LpLm1}:
	\begin{align*}
	\Lad{+}\Lad{-} &= (\sym{123} + i \sym{3})^2T_+T_- + (\sym{0} - 1/2)^2(T_+T_- - (\sym{0} - 1/2)^2 - (\sym{123} + i \sym{3})^2)\\
	&= \left((\sym{123}+i\sym{3})^2 + (\sym{0} - 1/2)^2\right)T_+T_- - (\sym{0}-1/2)^2\left((\sym{0}-1/2)^2 + (\sym{123} + i \sym{3})^2\right),
	\intertext{and because $\sym{0}$ commutes with $\sym{123}$ and $\sym{3}$, it factorises as}
	&= -\left( (\sym{0} - 1/2)^2+ (\sym{123} +  T_0)^2\right)\left((\sym{0} - 1/2)^2 - T_+T_-\right).
	\end{align*}
	
	The same process applies of course with the expression of $\sym{-}\sym{+}$ in Corollary~\ref{coro:sqO123} for equation~\eqref{eq:LmLp1}.	The two factorisations~\eqref{eq:factoLpLm} and~\eqref{eq:factoLmLp} have been exhibited, concluding the proof.
\end{proof}

Put  $\dtau:= \dcover{\sigma}_1\dcover{\sigma}_m$, $\dtau^{-1} =\dcover{\sigma}_m\dcover{\sigma}_1$ and as before $\zeta = e^{\pi i/m}$. From Lemma~\ref{lem:commWO0}, the actions of the reflections on the symmetries and on the ladder operators are given by
\begin{equation}\label{eq:actionofreflonOpOmandLpm}
\begin{aligned}
\dcover{\sigma}_0\sym{0} &= \sym{0}\dcover{\sigma}_0,  & \dcover{\sigma}_1\sym{0} &= -\sym{0}\dcover{\sigma}_1, & \dcover{\sigma}_m\sym{0} &= -\sym{0}\dcover{\sigma}_m,\\
\dcover{\sigma}_0\sym{\!+} &= -\sym{+}\dcover{\sigma}_0,  & \dcover{\sigma}_1\sym{+} &= \zeta^2\sym{-}\dcover{\sigma}_1, & \dcover{\sigma}_m\sym{+} &= \sym{-}\dcover{\sigma}_m,\\
\dcover{\sigma}_0\sym{-} &= -\sym{-}\dcover{\sigma}_0, & \dcover{\sigma}_1\sym{-} &= \zeta^{-2}\sym{+}\dcover{\sigma}_1, & \dcover{\sigma}_m\sym{-} &= \sym{+}\dcover{\sigma}_m,\\
\dcover{\sigma}_0\Lad{+} &= -\Lad{+}\dcover{\sigma}_0,  & \dcover{\sigma}_1\Lad{+} &= -\zeta^2\Lad{-}\dcover{\sigma}_1, & \dcover{\sigma}_m\Lad{+} &= -\Lad{-}\dcover{\sigma}_m,\\
\dcover{\sigma}_0\Lad{-} &= -\Lad{-}\dcover{\sigma}_0,  & \dcover{\sigma}_1\Lad{-} &= -\zeta^{-2}\Lad{+}\dcover{\sigma}_1, & \dcover{\sigma}_m\Lad{-} &= -\Lad{+}\dcover{\sigma}_m.\\
 \end{aligned}
\end{equation}
\begin{equation}\label{eq:actiontauandTonL}
\begin{aligned}
\dtau\Lad{+} &= \zeta^{-2}\Lad{+}\dtau, & 
 \dtau\Lad{-} &= \zeta^2\Lad{-}\dtau,\\
 \dtau^{-1}\Lad{+} &= \zeta^2\Lad{+}\dtau^{-1}, & 
 \dtau^{-1}\Lad{-} &= \zeta^{-2} \Lad{-}\dtau^{-1},\\
T_+ \Lad{-} &= \Lad{+}T_-, & T_-\Lad{+} &= \Lad{-}T_+.
\end{aligned}
\end{equation}

\subsection{Unitary structure}\label{ssec:unitarity}
The commutation relations of the two-index symmetries~\eqref{eq:commrelO0OpOm} of the algebra $\SymAlg$ reduce to the commutation relations of $\Lie{so}(3)$ (or of $\Lie{sl}(2)$) when the map $\kappa$ becomes zero. Any $*$-structure on $\SymAlg$ must then reduce to either $\Lie{su}(2)$ or $\Lie{su}(1,1)$, of which only $\Lie{su}(2)$  admits finite-dimensional unitary representations. Let $*:\SymAlg\to\SymAlg$ be the anti-linear ($(aX+bY)^* = \conj{a} X^* + \conj{b}Y^*$) anti-involution ($(XY)^* = Y^* X^*$, $(X^*)^* = X$) defined on generators by
\begin{equation}
	\begin{aligned}
		 \sym{0}^* &= \sym{0}, & \sym{\pm}^* &= \sym{\mp}, & \sym{123}^* &= -\sym{123}, & T_0^* &= -T_0, & T_{\pm}^* &= T_{\mp}, & \dsig{j}^* &= \dsig{j},
	\end{aligned}
\end{equation}
for any reflection $\dsig{j}\in \SymAlg$. As a consequence, the ladder operators satisfy $\Lad{\pm}^*= \Lad{\mp}$.

Direct computations show that the relations are compatible with the commutation relations~\eqref{eq:commrelO0OpOm}. Furthermore, sending $\kappa$ to 0 indeed gives the $*$-structure of $\Lie{su}(2)$, thus admitting finite-dimensional irreducible representations.

It would be possible to study unitarity with another structure, say imposing $\dsig{j}^{\dagger} = - \dsig{j}$, $\sym{\pm}^{\dagger} = -\sym{\mp}$ and $T_{\pm}^{\dagger} = -T_{\mp}$, but this would not include the important monogenics example.

\section{Finite-dimensional irreducible representations}\label{sec:Repr}
This section classifies the finite-dimensional irreducible representations of $\SymAlg$ and verifies if they are unitary under the $*$-structure presented in Section~\ref{ssec:unitarity}. The section is divided \modif{into three parts:} first the theorems are presented, then the idea of the proofs are exhibited, and the last part gives the details of the proof of Theorem~\ref{thm:irrepseven}. Since the proof of Theorem~\ref{thm:irrepsodd} is \modif{similar}, it is not included\footnote{It can be found in \modif{an older version of the paper:} \href{https://arxiv.org/abs/2010.03381v2}{arXiv:2010.03381v2}.}. 

\modif{The techniques employed share some similarities with the construction of standard modules for the representation theory of rational Cherednik algebras~\cite{chmutova_representations_2006,dezelee_representations_2003, rouquier2005representations}. We construct the irreducible finite-dimensional $\SymAlg$-representations from a certain class of representations of $\dcover{W}$.} For the benefit of the reader, the complete construction of \modif{the finite-dimensional irreducible representations of the group $\dcover{W}$} is included in Appendix~\ref{app:doublecover},  and they are presented in Theorem~\ref{thm:irrepsW}. \modif{The specific representations we need are those for which the commuting element $z\in\dcover{W}$ acts as $-\mathrm{id}$ ($\varepsilon=-1$ in the notation of Theorem~\ref{thm:irrepsW}); this comes from the realisation of the group $\dcover{W}$ in $\SymAlg$.} These specific representations are called \emph{spin representations}~\cite{Morris76}.

The existence of irreducible \modif{finite-dimensional $\SymAlg$-representations} , and their unitarity, is constrained by the map $\kappa$. For an integer $N$ and a certain irreducible spin representation $V$ of $\dcover{W}$, \modif{Theorem~\ref{thm:irrepsodd} ($m$ odd) and Theorem~\ref{thm:irrepseven} ($m$ even) present the conditions on $\kappa$ for the existence of an irreducible $\SymAlg$-representation \modif{$L_{\lambda,\Lambda}(V)$} of \modif{dimension} $2N+2$.} Note that the theorems take $V$ and $N$ independently. Even if the conditions constraining $\kappa$ depend on both of them, it means that the \modif{$\dcover{W}$-representation} $V$ does not fix the dimension.

The theorems adopt some notational conventions. First, $j,k\in \{1,\dots, N\}$. Second, $a\equiv_m b$ is a short-hand for $a\equiv b \modSB m$.  The tables are divided in families according to some conditions linked to $N$, the \modif{data} of $V$, and the values of $\lambda$ \modif{and $\Lambda$}.  When $m$ is even, we denote $m=2p$. The constants $\kappa_0$, $\kappa_1$ and $\kappa_m$ will be real and positive.  \modif{Finally, the indices of $\lambda$ in the tables indicate the several possibilities of the parameter.}



\begin{theorem}[Conditions for irreducible and unitary representations, the odd cases]\label{thm:irrepsodd}
	 Let $\kappa_0$ and $\kappa_1$ be positive constants. Let $N$ be a non-negative integer and let $V\simeq Y_{\ell}(-1,\delta)$ be an irreducible spin representation of $\dcover{W}$.  If the conditions on $\kappa_0$ and $\kappa_1$ presented in the next tables are respected, then $V$ extends to a $2N+2$ dimensional irreducible representation \modif{$\stdquo$} of $\SymAlg$. \modif{Moreover, the} constant  $\Lambda$ lies in one of the two families
\begin{equation}
    \begin{aligned}
	\Lambda_1 &\in \{\pm i(\modif{\lambda} + 1/2 + \kappa_0\delta)\} & &\text{or}& \Lambda_2 &\in \{\mp i(\modif{\lambda} +1/2 - \kappa_0\delta)\}.
	\end{aligned}
\end{equation}
 It is unitary if $\kappa_0$ and $\kappa_1$ satisfy further conditions. Furthermore, all irreducible finite-dimensional representations of $\SymAlg$ are of this form.
			\begin{equation*}
			\begin{array}{p{0.075\textwidth}|p{0.55\textwidth}|p{0.3\textwidth}}
			\multicolumn{3}{l}{\text{Case I: } 2(N+\ell)+1\equiv_m 0,\hfil \modif{ \lambda=\lambda_1 = N+1/2+\kappa_1m \hfill} }\\ \hline
			\hline
			$(\Lambda,\delta)$&\textrm{Irreducibility} & \textrm{Unitarity}\\
			\hline
			$(\Lambda_1,1)$,
			
			$(\Lambda_2,-1)$
			&
			No restriction
			&   
			No restriction
			\\
			\hline
			$(\Lambda_2,1)$,
			
			$(\Lambda_1,-1)$
			&
			$\kappa_0 \not\in \{k/2, \, \lambda_1+1/2-k/2\mid \text{odd }k\}$	 
			&
			$\kappa_0 <1/2$ or $\kappa_0 >\lambda_1$  \\
			\hline
			\end{array}
			\end{equation*}
			\begin{equation*}
	\begin{array}{p{0.075\textwidth}|p{0.55\textwidth}|p{0.3\textwidth}}
 \hline
	\multicolumn{3}{l}{ \text{Case I: } 2(N+\ell)+1\equiv_m 0,\hfil  \modif{\lambda = \lambda_2 = N+1/2-\kappa_1m}\hfill }\\ \hline
	\hline
	$(\Lambda,\delta)$&\textrm{Irreducibility} & \textrm{Unitarity}\\
	\hline
	$(\Lambda_1,1)$,
	
	$(\Lambda_2,-1)$
	&
	$\kappa_0 \not\in \left\{-\lambda_2-1/2+k/2\ \middle|\ \text{odd } k\right\}$
	
	$\kappa_1 \not\in \left\{\tfrac{N-k+1}{2m},\, \tfrac{N-j+1}{m}\ \middle|\ 2(k+\ell)\equiv_m 1,\, 2j+\ell\not\equiv_m -1\right\}$
	&   
	$\kappa_1< \tfrac{1}{2m}$
	\\
	\hline
	$(\Lambda_2,1)$,
	
	$(\Lambda_1,-1)$
	&
	$\kappa_0 \not\in \{k/2,\,  \lambda_2+1/2-k/2\mid \text{odd }k\}$
	
	$\kappa_1 \not\in \left\{\tfrac{N-k+1}{2m}, \tfrac{N-j+1}{m} \ \middle|\ 2(k+\ell)\equiv_m 1,\, 2j+\ell\not\equiv_m -1\right\}$ 	 
	&
	$\kappa_0 <1/2$ or $\kappa_0 >\lambda_2$
	
	$\kappa_1 <\tfrac{1}{2m}$ \\
	\hline
	\end{array}
	\end{equation*}
		\begin{equation*}
	\begin{array}{p{0.075\textwidth}|p{0.45\textwidth}|p{0.4\textwidth}}
	\multicolumn{3}{l}{\text{Case II: } 2(N+\ell)+1\not\equiv_m 0,\hfil  \modif{\lambda = \lambda_3= N+1/2} \hfill}\\ \hline\hline
	$(\Lambda,\delta)$&\textrm{Irreducibility} & \textrm{Unitarity}\\
	\hline
	$(\Lambda_1,1)$,
	
	$(\Lambda_2,-1)$
	&
$\kappa_1 \not\in \left\{ \tfrac{N+1-k}{
	m}\ \middle| \ 2(k+\ell)\equiv_m 1\right\}$ 
	&  $\kappa_1< \min\left( \tfrac{N+1-k}{m} \ \middle|\ 2k+\ell\equiv_m 1\right)$
	\\
	\hline
	$(\Lambda_2,1)$,
	
	$(\Lambda_1,-1)$
	&
$\kappa_0 \not\in\left\{k/2,\, \lambda_3 +1/2 - k/2\ \middle|\ \text{odd }k\right\}$

$\kappa_1 \not\in \left\{ \tfrac{N+1-k}{m}\ \middle|\ 2(k+\ell)\equiv_m 1\right\}$	 
	&
	$\kappa_0 <1/2$ or $\kappa_0 >\lambda_3$
	
	$\kappa_1< \min\left(\tfrac{N+1-k}{m}\ \middle|\ 2k+\ell \equiv_m 1\right)$\\
	\hline
	\end{array}
	\end{equation*}
		\begin{equation*}
	\begin{array}{p{0.075\textwidth}|p{0.45\textwidth}|p{0.4\textwidth}}
	\multicolumn{3}{l}{\text{Case III: even } N\hfil  \modif{\lambda = \lambda_4 = N/2+\kappa\delta} \hfill}\\ \hline\hline
	$(\Lambda,\delta)$&\textrm{Irreducibility} & \textrm{Unitarity}\\
	\hline
	$(\Lambda_1,1)$,
	
	$(\Lambda_2,-1)$
	&
	$\kappa_0 \not\in\left\{ \tfrac{2k-N-1}{2} \ \middle| \ k>N/2,\, 2(k+\ell)\not\equiv_m 1\right\}$
	
	$\kappa_1 \not\in \left\{ \tfrac{|\modif{\lambda_4}-k+1/2|}{m}
	\middle| \ 2(k+\ell)\equiv_m 1\right\}$ 
	&  $\kappa_1< \min\left(\tfrac{|N/2+ \kappa_0 + 1/2-k|}{m}\ \middle|\ (2k+\ell)\equiv_m 1\right)$
	\\
	\hline
	$(\Lambda_2,1)$,
	
	$(\Lambda_1,-1)$
	&
	$\kappa_0 \not\in\left\{\tfrac{k}{2},\, \tfrac{N-k+1}{4},\, \tfrac{N+1-2j}{2}\ \middle|\ \text{odd }k,\; \; 2(j+\ell)\not\equiv_m 1\right\}$ 
	
	$\kappa_1 \not\in \left\{ \tfrac{|\modif{\lambda_4}-k+1/2|}{m}
	\middle| \ 2(k+\ell)\equiv_m 1\right\}$ 
	&
	$\kappa_0 <1/2$ or $\kappa_0 >N/2+1/2$
	
	$\kappa_1< \min\left(\tfrac{|N/2+\kappa_0 + 1/2-k|}{m} \ \middle|\ (2k+\ell) \equiv_m 1 \right)$\\
	\hline
	\end{array}
	\end{equation*}
\end{theorem}

\begin{theorem}[Conditions for irreducible and unitary representations, the even cases]\label{thm:irrepseven}
Denote $m=2p$ for a certain $p\in\NN$. Let $\kappa_0$, $\kappa_1$ and $\kappa_m$ be positive constants. Let $N$ be a non-negative integer and $V\simeq Y_{2\ell+1}(-1,\delta)$ be an irreducible spin representation of $\dcover{W}$. If the conditions on $\kappa_0$, $\kappa_1$ and $\kappa_m$ of the following tables hold, then $V$ extends to a $2N+2$ dimensional irreducible $\SymAlg$-representation \modif{$\stdquo$}. \modif{Moreover, the} constant $\Lambda$ is of the form
\begin{align}
        \Lambda_1 &\in \{\pm i(\modif{\lambda}+1/2 +\kappa_0 \delta)\} &\text{or}&&\Lambda_2 &\in \{\mp i(\modif{\lambda} +1/2 - \kappa_0\delta)\}.
\end{align}
   It is unitary if $\kappa_0$, $\kappa_1$ and $\kappa_m$ satisfy more restrictive conditions presented thereafter. Furthermore, all irreducible finite-dimensional representations of $\SymAlg$ are of this form.

	\begin{equation*}
\begin{array}{p{0.075\textwidth}|p{0.5\textwidth}|p{0.35\textwidth}}
\multicolumn{3}{l}{\text{Case I.i: } N+\ell\equiv_m 1-p, \hfil \modif{\lambda = \lambda_1 = N+1/2+(\kappa_1+\kappa_m)p}\hfill }\\ \hline \hline
$(\Lambda,\delta)$&\textrm{Irreducibility} & \textrm{Unitarity}\\
\hline
$(\Lambda_1,1)$,

$(\Lambda_2,-1)$
&
No restriction
&
No restriction \\
\hline
$(\Lambda_2,1)$,

$(\Lambda_1,-1)$
&
$\kappa_0 \not\in \{k/2, \, \lambda_1+1/2-k/2\mid \text{odd }k\}$
&
$\kappa_0 <1/2,$ or $\kappa_0 >\lambda_1$  \\
\hline
\end{array}
\end{equation*}
	\begin{equation*}
\begin{array}{p{0.075\textwidth}|p{0.5\textwidth}|p{0.35\textwidth}}
 \hline
\multicolumn{3}{l}{\modif{\text{Case I.i: } N+\ell\equiv_m 1-p, } \hfil    \modif{\lambda = \lambda_2 = N+1/2-(\kappa_1+\kappa_m)p}\hfill }\\
\hline\hline
$(\Lambda,\delta)$&\textrm{Irreducibility} & \textrm{Unitarity}\\
\hline
$(\Lambda_1,1)$,

$(\Lambda_2,-1)$
&
 $\kappa_0 \not\in \left\{-\lambda_2-1/2+k/2\ \middle|\ \text{odd } k\right\}$

$\kappa_1,\, \kappa_m \not\in \left\{\tfrac{N-k+1}{m} \ \middle|\ 2 -k-\ell\equiv_m 0 \right\}$

$\kappa_1+\kappa_m \not\in \left\{\tfrac{N-k+1}{m},\, \tfrac{N-j+1}{p}\ \middle|\ 2-k-\ell \equiv_m p,\, j \not\equiv_m 0,p\right\}$
&
$\kappa_1+\kappa_m< 1/p$ \\
\hline
$(\Lambda_2,1)$,

$(\Lambda_1,-1)$
&
$\kappa_0 \not\in \left\{k/2,\,  \lambda_2+1/2-k/2\ \middle|\ \text{odd }k\right\}$

$\kappa_1,\, \kappa_m \not\in \left\{\tfrac{N-k+1}{m} \ \middle|\ 2 -k-\ell\equiv_m 0 \right\}$

$\kappa_1+\kappa_m \not\in \left\{\tfrac{N-k+1}{m},\, \tfrac{N-j+1}{p}\ \middle|\ 2-k-\ell \equiv_m p,\, j \not\equiv_m 0,p \right\}$
&
$\kappa_0 <1/2$ or $\kappa_0 >\lambda_2$

$\kappa_1+\kappa_m <1/p$    \\
\hline
\end{array}
\end{equation*}
			\begin{equation*}
	\begin{array}{p{0.075\textwidth}|p{0.5\textwidth}|p{0.35\textwidth}}
	\multicolumn{3}{l}{\text{Case I.ii: } N+\ell\equiv_m 1, \hfil   \modif{\lambda = \lambda_3 = N+1/2+(\kappa_1-\kappa_m)p} \hfill }\\ \hline\hline
	$(\Lambda,\delta)$&\textrm{Irreducibility} & \textrm{Unitarity}\\
	\hline
	$(\Lambda_1,1)$,

	$(\Lambda_2,-1)$
	&
	$\kappa_m \not\in \left\{\tfrac{N-k+1}{m} \ \middle|\ 2 -k-\ell\equiv_m p\right\}$ 
	&
	$\kappa_m< \min\left(\tfrac{N-k+1}{m}\ \middle|\ 2-k-\ell \equiv_m p\right)$ \\
	\hline
	$(\Lambda_2,1)$,
	
	$(\Lambda_1,-1)$
	&
	$\kappa_0 \not\in \{k/2,\, \lambda_3+1/2-k/2\mid \text{odd }k\}$
	
	$\kappa_m \not\in \left\{\tfrac{N-k+1}{m}\ \middle|\ 2 -k-\ell\equiv_m p \right\}$
	&
	$\kappa_0 <1/2$ or $\kappa_0 >\lambda_3$
	
	$\kappa_m< \min\left(\tfrac{N-k+1}{m}\ \middle|\ 2-k-\ell \equiv_m p\right)$ \\
	\hline
	\end{array}
	\end{equation*}
	\begin{equation*}
	\begin{array}{p{0.075\textwidth}|p{0.5\textwidth}|p{0.35\textwidth}}
	\hline
	\multicolumn{3}{l}{\modif{\text{Case I.ii: } N+\ell\equiv_m 1,} \hfil  \modif{\lambda = \lambda_4 = N+1/2-(\kappa_1-\kappa_m)p} \hfill}\\
	\hline\hline
	$(\Lambda,\delta)$&\textrm{Irreducibility} & \textrm{Unitarity}\\
	\hline
		$(\Lambda_1,1)$,
	
	$(\Lambda_2,-1)$
	&
	 $\kappa_1 \not\in \left\{\tfrac{N-k+1}{m}\ \middle| \ 2 -k-\ell\equiv_m p\right\}$ 
	
	$\kappa_1-\kappa_m \not\in \left\{\tfrac{N-k+1}{m},\, \tfrac{N-j+1}{p}\ \middle|\ {2-k-\ell \equiv_m 0,\atop 2-j-\ell\not\equiv_m 0,p} \right\}$ 
	&
	  $\kappa_1 - \kappa_m< 1/p$ 
	
	$\kappa_1 < \min\left(\tfrac{N-k+1}{m}\ \middle| \ 2-k-\ell \equiv_m p\right)$\\
	\hline
	$(\Lambda_2,1)$,
	
	$(\Lambda_1,-1)$
	&
	$\kappa_0 \not\in \{k/2,\, \lambda_4+1/2-k/2\mid \text{odd }k\}$ 
	
	$\kappa_1 \not\in \left\{\tfrac{N-k+1}{m} \ \middle|\ 2 -k-\ell\equiv_m p \right\}$ 
	
	$\kappa_1-\kappa_m \not\in \left\{\tfrac{N-k+1}{m},\ \tfrac{N-j+1}{p}\ \middle|\ {2-k-\ell \equiv_m 0,\atop 2-j-\ell\not\equiv_m 0,p} \right\}$
	&
	$\kappa_0 <1/2$ or $\kappa_0 >\lambda_4$ 
	
	$\kappa_1 - \kappa_m< 1/p$ 
	
	$\kappa_1<\min\left(\tfrac{N-k+1}{m}\ \middle|\ 2-k-\ell \equiv_m p\right)$\\
	\hline
	\end{array}
	\end{equation*}
	
		\begin{equation*}
	\begin{array}{p{0.075\textwidth}|p{0.45\textwidth}|p{0.400\textwidth}}
	\multicolumn{3}{l}{\text{Case II: } (N+\ell)\not\equiv_m 1,\,1-p \hfil  \modif{\lambda = \lambda_5= N+1/2} \hfill }\\\hline \hline
	$(\Lambda,\delta)$&\textrm{Irreducibility} & \textrm{Unitarity}\\
	\hline
	$(\Lambda_1,1)$,
	
	$(\Lambda_2,-1)$
	&
	$\kappa_1 + \kappa_m \not\in \left\{\tfrac{N-k+1}{p}\ \middle|\ 2 -k-\ell\equiv_m p \right\}$ 
	
	$|\kappa_1-\kappa_m| \not\in \left\{\tfrac{N-k+1}{p}\ \middle|\ 2-k-\ell \equiv_m 0 \right\}$ 
	&  $\kappa_1 + \kappa_m < \min\left(\tfrac{N-k+1}{m}\ \middle|\ 2-k-\ell\equiv_m p\right)$
	
	$|\kappa_1-\kappa_m| < \min\left(\tfrac{N-k+1}{p}\ \middle|\ 2-k-\ell \equiv_m 0\right)$
	\\
	\hline
	$(\Lambda_2,1)$,
	
	$(\Lambda_1,-1)$
	&
	$\kappa_0 \not\in \{k/2,\, \lambda_5+1/2-k/2\mid \text{odd }k\}$ 
	
	$\kappa_1 + \kappa_m \not\in \left\{\tfrac{N-k+1}{p}\ \middle|\ 2 -k-\ell\equiv_m p \right\}$ 
	
	$|\kappa_1-\kappa_m| \not\in \left\{\tfrac{N-k+1}{p}\ \middle|\ 2-k-\ell \equiv_m 0 \right\}$ 
	&
	$\kappa_0 <1/2$ or $\kappa_0 >N-1/2$ 
	
	$\kappa_1 + \kappa_m< \min\left(\tfrac{N-k+1}{m}\ \middle|\ 2-k-\ell\equiv_m p\right)$
	
	$|\kappa_1-\kappa_m| < \min\left(\tfrac{N-k+1}{p}\ \middle|\ 2-k-\ell \equiv_m 0\right)$ \\
	\hline
	\end{array}
	\end{equation*}
		\begin{equation*}
\begin{array}{p{0.075\textwidth}|p{0.47\textwidth}|p{0.42\textwidth}}
\multicolumn{3}{l}{\text{Case III: even } N\hfil  \modif{\lambda = \lambda_6 = N/2+\kappa_0\delta} \hfill}\\ \hline\hline
$(\Lambda,\delta)$&\textrm{Irreducibility} & \textrm{Unitarity}\\
\hline
$(\Lambda_1,1)$,

$(\Lambda_2,-1)$
&
 $\kappa_0 \not\in \left\{ \tfrac{N-2k+1}{2} \ \middle|\ 2-k-\ell \not\equiv_m 0,p\right\}$

$\kappa_1 + \kappa_m \not\in \left\{\tfrac{|\lambda_6-k+1/2|}{p}\ \middle|\ 2 -k-\ell\equiv_m p \right\}$ 

$|\kappa_1-\kappa_m| \not\in \left\{\tfrac{|\lambda_6-k+1/2|}{p} \ \middle|\ 2-k-\ell \equiv_m 0 \right\}$
&  $\kappa_1 + \kappa_m< \min\left(\tfrac{|\lambda_6-k+1/2|}{p}\ \middle| \ 2-k-\ell\equiv_m p\right)$

$|\kappa_1-\kappa_m| < \min\left( \tfrac{|\lambda_6-k+1/2|}{p}\ \middle|\ 2-k-\ell \equiv_m 0\right)$
\\
\hline
$(\Lambda_2,1)$,

$(\Lambda_1,-1)$
&
$\kappa_0 \not\in \left\{\tfrac{k}{2},\, \tfrac{N-k+1}{4},\, \tfrac{2j-N-1}{2}\ \middle|\ \text{odd }k, \; 2-j-\ell \not\equiv_m 0,p\right\}$

$\kappa_1 + \kappa_m \not\in \left\{\tfrac{|\lambda_6-k+1/2|}{p}\ \middle|\ 2 -k-\ell\equiv_m p \right\}$

$|\kappa_1-\kappa_m| \not\in \left\{\tfrac{|\lambda_6-k+1/2|}{p}\ \middle|\ 2-k-\ell \equiv_m 0 \right\}$
&
$\kappa_0 <1/2$ or $\kappa_0 >N/2+1/2$

$\kappa_1 + \kappa_m< \min\left(\tfrac{|\lambda_6-k+1/2|}{p}\ \middle|  \ 2-k-\ell\equiv_m p\right)$

$|\kappa_1-\kappa_m| < \min \left(\tfrac{|\lambda_6-k+1/2|}{p} \ \middle|\ 2-k-\ell \equiv_m 0\right)$ \\
\hline
\end{array}
\end{equation*}
\end{theorem}

\subsection{Preliminary general results and idea of the proofs}\label{ssec:proofidea}

The proofs are straightforward, but long. They are constructive: in doing them, all the \modif{finite-dimensional} representations are found, and the conditions are naturally derived from the constructions. The idea behind them is akin to the standard module construction, so the first step is to study representations of $\SymAlg$ by starting from an irreducible spin representation of the group $\dcover{W}$. \modif{Note that unlike a semisimple Lie algebra or a rational Cherednik algebra, the algebra $\SymAlg$ does not have a triangular decomposition because the action of the group $\dcover{W}$ interchanges $\Lad{+}$ and $\Lad{-}$, see equations~\eqref{eq:actionofreflonOpOmandLpm} and~\eqref{eq:actiontauandTonL}. However, let $\dcover{W}_0$ denote the subgroup of $\dcover{W}$ generated by the elements commuting with $\sym{0}$. Then the associative subalgebra of $\SymAlg$ generated by $\Lad{-}$, $\Lad{+}$, $\sym{0}$, $\sym{123}$ and $\dcover{W}_0$ does exhibit a triangular decomposition.} 

\modif{Let  $V$ be an irreducible $\dcover{W}$-representation. It decomposes into irreducible $\dcover{W}_0$-representations, and thus specifically for the case at hand, into two one-dimensional representations, see Appendix~\ref{app:doublecover}. The elements $\sym{0}$ and $\sym{123}$ commute with $\dcover{W}_0$ and they act thus by scalar multiplication on $\dcover{W}_0$-representations.}

\modif{From there, we use the triangular decomposition of the subalgebra generated by $\Lad{-}$, $\Lad{+}$, $\sym{0}$, $\sym{123}$ and $\dcover{W}_0$ and then work out the action of the rest of the symmetry algebra $\SymAlg$.}

\modif{We show that all the representations  $L_{\lambda,\Lambda}(V)$ of Theorems~\ref{thm:irrepsodd} and~\ref{thm:irrepseven} are obtained from the sets of eigenvectors given by Lemma~\ref{lem:inducrep}, and we give the restrictions on the function $\kappa$ by examining the action of the ladder operators on them. As we cover all the possible cases, a complete set of finite-dimensional irreducible representations of $\SymAlg$ is exhibited.}

\begin{lemma}\label{lem:inducrep}
	Let $\mathcal{V}$ be a finite-dimensional irreducible representation of $\SymAlg$. There exists a set of eigenvectors of $\sym{0}$ and $\sym{123}$ 
	\begin{equation}
		\mathcal{B} = \left\{ v_k^+, \, v_k^- \ \middle| \ 0\leq k \leq N\right\}
	\end{equation}
that generates $\mathcal{V}$\modif{,  with each pair $\langle v_k^-,\,v_k^+\rangle$ generating an irreducible spin representation of $\dcover{W}$ and $\Lad{+}v_0^+ =0$, $\Lad{-}v_N^+ =0$.}
\end{lemma}
\begin{proof}
\modif{We begin by decomposing $\mathcal{V}$ into irreducible spin representations of $\dcover{W}$ and exhibit an $\sym{0}$- and $\sym{123}$-eigenvector $v_0^+$ from the further decomposition into $\dcover{W}_0$-representations that satisfies the condition of the lemma . We then show that putting $v_k^+:= \Lad{-}^kv_0^+$ and $v_k^- := \dsig{m}v_k^+$ in the set $\mathcal{B}$ proves the lemma.}

As $\mathcal{V}$ is a $\SymAlg$-representation, it is also a $\dcover{W}$-representation. Furthermore, in its realisation in $\SymAlg$, a representation of $\dcover{W}$ must be a spin representation: abstractly $\dcover{W}$ has a commuting element $z$ that acts as $-1$ or $+1$ on the representation, but the realisation forces $z$ to act as $-1$.   By Maschke's Theorem, the $\dcover{W}$-representation $\mathcal{V}$ is expressible as a direct sum of irreducible spin representations of  $\dcover{W}$. From Theorem~\ref{thm:irrepsW}, all the irreducible spin representation of $\dcover{W}$ are two-dimensional. 

\modif{Each of the irreducible spin representations further decomposes as the sum of two one-dimensional irreducible $\dcover{W}_0$-representations. Since $\sym{0}$ and $\sym{123}$ commute with $\dcover{W}_0$ they act as multiple of the identity on the generators of the $\dcover{W}_0$-representations. Let $v$ be any such generator. The element $\dsig{m}v$ generates another $\dcover{W}_0$-representation and the pair $(\dsig{m} v, v)$ generates an irreducible spin $\dcover{W}$-representation. As $\Lad{+}$ and $\Lad{-}$ form a pair of ladder operators with respect to $\sym{0}$, we have that 
\begin{equation}
\sym{0}\Lad{\pm}^kv = \left(\comm{\sym{0},\Lad{\pm}^k} + \Lad{\pm}^k \sym{0}\right)v = (\pm k \Lad{\pm}^k + \Lad{\pm}^k\sym{0})v.
\end{equation}
Hence, $\Lad{+}^kv$ and $\Lad{-}^kv$ are also eigenvectors of $\sym{0}$, with their eigenvalues respectively raised or lowered by $k$. Since $\mathcal{V}$ is finite-dimensional, one of the generators in the $\dcover{W}_0$-decomposition must be annihilated by $\Lad{+}$; denote it by $v_0^+$. Let $\lambda$ and $\Lambda$ be the eigenvalues for $\sym{0}$ and $\sym{123}$ of this element
\begin{align}
 \sym{0} v_0^+ &= \lambda v_0^+, & \sym{123}v_0^+ &= \Lambda v_0^+.
 \end{align}
 Applying $\Lad{-}$ lowers the eigenvalue and changes the $\dcover{W}_0$-representation. In particular, $v_k^+ = \Lad{-}v_0^+$ is an eigenvector of $\sym{0}$ of eigenvalue $\lambda - k$
 \begin{equation}
 	\sym{0}v_k^+ = (- k \Lad{-}^k + \Lad{-}^k\sym{0})v_0^+ = (\lambda-k)v_k^+.
 \end{equation}
}

 \modif{As all the eigenvalues are distinct, there must be a $N$ such that $\Lad{-}v_N^+=0$ since the representation is finite-dimensional. Hence, we have exhibited two elements $v_0^+$ and $v_N^+$ satisfying}
\begin{equation}\label{eq:highestweight}
 \Lad{+}v_0^+ = 0, \qquad  \modif{\Lad{-}v_N^+ = \Lad{-}^{N+1} v_0^+ =0.}
\end{equation}

Furthermore,  $v_k^- = \dsig{m} v_k^+$ is also an eigenvector of $\sym{0}$ \modif{and $\sym{123}$}, indeed
\begin{align}\label{eq:O0eigenval}
\sym{0} v_k^- &= -\dsig{m}\sym{0} v_k^+ = -(\lambda - k)\dsig{m}v_k^+ = (k-\lambda)v_k^{-}, & \modif{\sym{123}v_k^-} &= \modif{\dsig{m}\sym{123}v_k^+ = \Lambda v_k^-.}
\end{align}

There is thus a  set of eigenvectors of $\sym{123}$ and $\sym{0}$
\begin{equation}
\mathcal{B} = \{v_k^+ \mid k=0,\dots, N\} \cup \{v_k^- \mid k =0, \dots , N\}.
\end{equation}
It is a generating set of $\mathcal{V}$ because $\mathcal{V}$ is irreducible.
\end{proof}

Remark that the eigenvectors $v_k^+$ all have distinct $\sym{0}$-eigenvalues; it is however possible that $v_j^-$ has the same eigenvalue as one $v_k^+$. Note also that $\Lad{+}^kv_0^- = (-1)^k v_k^-$ and so we can also write $v_k^- = (-1)^k \Lad{+}^kv_0^-$.
 
\modif{The proof of the lemma gives for each finite-dimensional $\SymAlg$-representation a set of data: two eigenvalues $\lambda, \Lambda$ and a spin $\dcover{W}$-representation $V=\langle v_0^+,v_0^-\rangle$. Hence, we denote irreducible representations of $\SymAlg$ by $\stdquo$.
}

\modif{Note however that the lemma does not impose a unique choice of  data $(\lambda,\Lambda,V)$ to identify the representation $\mathcal{V}$. Indeed, we see immediately from~\eqref{eq:O0eigenval} that there could have been another choice for $v_0^+$ and $v_N^+$ since
\begin{equation}\label{eq:otherv0}
	\begin{aligned}
		\Lad{+}v_N^- &= 0, & \Lad{-}v_0^- &=0,
    \end{aligned}
\end{equation}
and the two sets of data $(\lambda,\Lambda,V)$ and $(N-\lambda,\Lambda, V'=\langle v_N^-,v_N^+\rangle)$ refer to the same $\SymAlg$-representation. This is taken into account into the classification, see the cases to solve the system~\eqref{eq:repevencond}. }

Furthermore on \modif{ a representation $\stdquo$}, we define unitarity with respect to the $*$-structure from Section~\ref{ssec:unitarity}. We define abstractly on $\modif{\stdquo}$ a sesquilinear form
\begin{equation}
	\bili{-}{-}: \modif{\stdquo} \times \modif{\stdquo} \to \CC
\end{equation}
\modif{extending the unitary structure of $V$ normalized by  $\bili{v_0^+}{v_0^+} = 1,$} such that, for all $X\in \SymAlg$ and $v,w\in \modif{\stdquo}$,
\begin{equation}
	\bili{Xv}{w} = \bili{v}{X^*w}.
\end{equation}

The next lemma gives a condition on unitarity assuming a specific form for the action of $\Lad{+}$. (It will be proven below that indeed $\Lad{+}$ acts like this.)
\begin{lemma}[Unitarity condition]\label{lem:unitaritycond}
	If $\Lad{+}$ acts on $v_k^+$ as	$\Lad{+}v_k^+ = A(k)v_{k-1}^+$, for certain constants $A(k)$, then $\modif{\stdquo}$ is unitary when $A(k)>0$ for $1\leq k \leq N$.
\end{lemma}
\begin{proof}
	As $\sym{0}^* = \sym{0}$, and because of the definition of $v_k^+$,
\begin{gather*}
	\bili{v_k^+}{v_{l}^+} = h_k \delta_{k,\, l} = \bili{v_k^-}{v_l^-},\\
	\bili{v_k^+}{v_l^-} = 0, \quad \bili{v_0^+}{v_0^+} = h_0 = 1.
\end{gather*}
The fact that $\Lad{\pm}^* = \Lad{\mp}$ gives a recursive structure for the $h_k$ linked with $A(k)$:
\begin{align*}
	h_k:= \bili{v_k^+}{v_k^+} =  \bili{\Lad{-}v_{k-1}^+}{v_k^+}= \bili{v_{k-1}^+}{\Lad{+}v_k^+} = \bili{v_{k-1}^+}{A(k)v_{k-1}^+} = A(k)h_{k-1}.
\end{align*}
Therefore, to have an inner product and unitarity, it must be that $A(k)>0$ for $1\leq k \leq N$.
\end{proof}

Assume that $\Lad{+}v_k^+ = A(k)v_{k-1}^+$ and $\Lad{-}v_k^- = A(k)v_{k-1}^-$ for certain $A(k)$ (as will be proved later). For both even and odd case, the proof of Lemma~\ref{lem:unitaritycond} indicates that there is an orthonormal basis, provided the representation is unitary.

\subsection{Proof of Theorem~\ref{thm:irrepseven}}\label{ssec:proofeven}

We prove that the set $\mathcal{B}$ of Lemma~\ref{lem:inducrep} is a basis of a \modif{$(2N+2)$-dimensional} irreducible representation of $\SymAlg$ characterized by two constants $\Lambda$ and $\lambda$ under the conditions on $\kappa$ of Theorem~\ref{thm:irrepseven}.

Let $m=2p$, \modif{$N\in \NN$}, $\ell\in\{0,1,\dots,p-1\}$ and $\delta \in \{-1,+1\}$. Put $V=Y_{2\ell+1}(\delta)$ an irreducible spin representation of $\dcover{W}$. \modif{By Lemma~\ref{lem:inducrep} and the discussion following it, consider the $\SymAlg$-representation  $\stdquo$} with its generating set of $2N+2$ eigenvectors of $\sym{0}$ and $\sym{123}$
\begin{equation}
 \mathcal{B} = \{v_k^{+} := \Lad{+}^k v_0^+,\ v_k^-:= \dsig{m}v_k^+\mid k = 0, \dots , N\}.
\end{equation}
The representation $V$ is generated by $v_0^+$ and $v_0^-$ and the $\sym{0}$- and $\sym{123}$-eigenvalues on $v_k^{\pm}$ are
\begin{align}
	\sym{0}v_k^{\pm} &= \pm(\lambda -k)v_k^{\pm}, & \sym{123}v_k^{\pm} &= \Lambda v_k^{\pm}.
\end{align}

The actions of the group elements on $v_k^{\pm}$ are given below. Recall that $\zeta = e^{\pi i /m} $:
\begin{align*}
\dtau v_k^+ &=  \zeta^{2k} \Lad{-}^k \dtau v_0^+ = \zeta^{2(k+\ell)+1} v_k^+,& \dtau v_k^- &= \zeta^{-2(k+\ell)-1} v_k^-,\\
\dsig{0}v_k^+ &= (-1)^k \Lad{-}^k\dsig{0} v_0^+ = (-1)^k \delta v_k^+,&\dsig{0} v_k^- &=  (-1)^{2k} \Lad{+}^k \dsig{0}v_0^- = (-1)^{k+1} \delta v_k^-,\\
\dsig{m} v_k^+ &=  (-1)^k  \Lad{+}^k \dsig{m} v_0^+ = v_k^-, & \dsig{m} v_k^- &= v_k^+.
\end{align*}

Recall that $T_0 = i\kappa_0\dsig{0}$ and so its action on $v_k^{\pm}$ is simply $T_0 v_k^{\pm} = \pm i(-1)^k \kappa_0 \delta v_k^{\pm}$. The actions of $T_+$ and $T_-$ are obtained from $\dsig{j} = (-1)^{j+1}\dtau^j\dsig{m}$ and equation~\eqref{eq:Tpmeven}, because $m$ is even:
\begin{align*}
T_{\pm} v_k^{+} &= \mp i(\kappa_1 T_{\pm}^1 + \kappa_m T_{\pm}^2)v_k^{+}, & 
T_{\pm} v_k^{-} &= \mp i(\kappa_1 T_{\pm}^1 + \kappa_m T_{\pm}^2)v_k^{-}.
\end{align*} 
On the odd root components $T_+^1$ and $T_-^1$, the action is given by
\begin{align*}
T_+^1 v_k^{\pm} &= \sum_{j=1}^p \zeta^{2j-1} \dsig{2j-1} v_k^{\pm} 
= \sum_{j=1}^{p} \zeta^{2j-1}(-1)^{2j} \dtau^{2j-1} \dsig{m} v_k^{\pm}
= \sum_{j=1}^p \zeta^{(2j-1)(1\mp 1 \mp 2(k+\ell) )} v_k^{\mp},\\
T_-^1 v_k^{\pm} &=  \sum_{j=1}^p \zeta^{(2j-1)(-1\mp 1 \mp 2(k+\ell) )} v_k^{\mp},
\intertext{and on the even roots components $T_+^2$ and $T_-^2$, by}
T_+^2 v_k^{\pm} &= \sum_{j=1}^p \zeta^{2j+1} \dsig{2j} v_k^{\pm} = \sum_{j=1}^{p} (-1)^{2j} \zeta^{2j} \dtau^{2j} \dsig{m} v_{k}^{\pm}
= -\sum_{j=1}^p \zeta^{2j(1\mp 1 \mp 2(k+\ell))} v_k^{\mp},\\
T_-^2 v_k^{\pm} &= -\sum_{j=1}^p \zeta^{2j(-1\mp 1 \mp 2(k+\ell))} v_k^{\mp}.
\end{align*}

Define 
\begin{align}
G_{\kappa}(X) &:= p( \kappa_1 1_m'(X) - \kappa_m 1_p(X)); & 1_p(X) 
&:= 
\begin{cases}
1, & X\equiv_p 0;\\
0, & \text{else};
\end{cases}
& 1'_m(X) 
&:= 
\begin{cases}
-1, & X\equiv_m p;\\
1, & X \equiv_m 0;\\
0, & \text{else}.
\end{cases}
\end{align}
The actions of $T_{+}$ and $T_-$ on $v_k^{-}$ and $v_k^+$ are then expressed with this shorthand notation as
\begin{align}
	T_{+} v_k^- &= -i G_{\kappa}(1-k-\ell) v_k^+,  & T_{+} v_k^+ &= -iG_{\kappa}(k+\ell)v_k^-,\\
	T_- v_k^- &= i G_{\kappa}(k+\ell)v_k^+, & T_-v_k^+ &= i G_{\kappa}(1-k-\ell) v_k^-.
\end{align}

For ease of notation, define
\begin{equation}
\begin{aligned}
        H_{\kappa}(x):= p^2((\kappa_1^2 + \kappa_m^2)1_p(x) - 2\kappa_1\kappa_m 1_m'(x))
        &= \begin{cases}
           	p^2(\kappa_1+\kappa_m)^2,& x \equiv_m p;\\
           	p^2(\kappa_1- \kappa_m)^2, & x\equiv_m 0;\\
           	0, & \text{else.}
           \end{cases}
\end{aligned}
\end{equation}
The actions of $T_+T_-$ and $T_-T_+$ on $v_k^+$ and $v_k^-$ are given below
\begin{align}
	T_-T_+ v_k^+ &= H_{\kappa}(k+\ell) v_k^+, &
	T_-T_+ v_k^- &= H_{\kappa}(1-k-\ell) v_k^-,\\
	T_+T_- v_k^+ &= H_{\kappa}(1-k-\ell) v_k^+,&
	T_+T_- v_k^- &= H_{\kappa}(k+\ell) v_k^-.
\end{align}

We now employ the factorisations~\eqref{eq:factoLpLm} and~\eqref{eq:factoLmLp} to get conditions on the actions of $\Lad{+}$ and $\Lad{-}$ on the vectors $v_k^+$ and $v_k^-$:
\begin{align*}
	\Lad{+}v_k^+ &= \Lad{+}\Lad{-}v_{k-1}^+ = -\left((\sym0 - 1/2)^2 + (\sym{123}+i\sym3)^2\right)\left((\sym0 -1/2)^2 - T_+T_-\right)v_{k-1}^+ \\
	&= -\left( \left( (\lambda - k + 1/2)^2 + (\Lambda + (-1)^{k-1}i \kappa_0\delta)^2\right) \left((\lambda - k + 1/2)^2 - H_{\kappa}(2-(k+\ell))\right)  \right)v_{k-1}^+,
	\intertext{and}
	\Lad{-}v_k^- &= -\Lad{-}\Lad{+}v_{k-1}^- = \left((\sym0 + 1/2)^2 + (\sym{123}-i\sym3)^2\right)\left((\sym0 + 1/2)^2 - T_-T_+\right) v_{k-1}^-\\
	&=\left( \left(( k -\lambda - 1/2 )^2 + (\Lambda + (-1)^{k-1} i\kappa_0 \delta)^2\right)\left((k -\lambda - 1/2)^2 - H_{\kappa}(2-(k+\ell))\right) \right) v_{k-1}^-.
\end{align*}
Put 
\begin{gather}
	A(k) := -\underbrace{\big((\lambda - k + 1/2)^2 + (\Lambda - (-1)^{k}i \kappa_0\delta)^2\big)}_{:=A_{(1)}(k)}\underbrace{\big((\lambda - k + 1/2)^2 - H_{\kappa}(2-k-\ell)\big)}_{:=A_{(2)}(k)}.
\end{gather}
So the actions are simply 
\begin{align}
	\Lad{+}v_k^+ &= A(k) v_{k-1}^{+}, & \Lad{-}v_k^- &= -A(k) v_{k-1}^ {-},\\
	\Lad{+}\Lad{-}v_k^- &= A(k)v_k^-, & \Lad{-}\Lad{+}v_k^+ &= A(k) v_k^+. 
\end{align}
The actions of $\sym{+}$ and $\sym{-}$ follow:
	\begin{align}
		\sym{-}v_k^+ &= {v_{k+1}^+ - 2i(\Lambda - (-1)^{k+1}i \kappa_0\delta) G_{\kappa}(1-k-\ell)v_k^- \over \lambda - k -1/2},& \lambda &\neq k+1/2, \label{eq:Omvpeven}\\
		\sym{-} v_k^- &= {-A(k)v_{k-1}^- - 2i(\Lambda - i(-1)^k \kappa_0 \delta)G_{\kappa}(k+\ell)v_k^+ \over k-\lambda - 1/2}, & \lambda &\neq 1/2-k,\label{eq:Omvmeven}\\
		\sym{+}v_k^+ &= { A(k)v_{k-1}^+ + 2i(\Lambda +i(-1)^{k+1} \kappa_0 \delta) G_{\kappa}(k+\ell)v_k^-\over \lambda -k +1/2}, & \lambda &\neq k-1/2,\label{eq:Opvpeven}\\
		\sym{+}v_k^- &= { -v_{k+1}^- + 2i(\Lambda + i (-1)^k \kappa_0 \delta) G_{\kappa}(1-k-\ell) v_k^+\over k-\lambda + 1/2}, & \lambda &\neq -k-1/2.\label{eq:Opvmeven}
	\end{align}

The system to solve for the irreducibility of $\modif{\stdquo}$ is
\begin{equation}\label{eq:repevencond}
	\begin{cases}
    \left( (\lambda  + 1/2)^2 + (\Lambda  - i \kappa_0\delta)^2\right) \left((\lambda  + 1/2)^2 - H_{\kappa}(2-\ell)\right) = 0,\\
	\ \left(( \lambda -N  - 1/2 )^2 + (\Lambda + (-1)^{N} i\kappa_0 \delta)^2\right)\left((\lambda - N - 1/2)^2 - H_{\kappa}(1-N-\ell)\right)= 0,\\
	A(k) \neq 0, \quad 1\leq k \leq N.
	\end{cases}
\end{equation}

The values of $\ell$ and $N$ influence the value of $H_{\kappa}(1-N-\ell)$ and justify the division in the following types. All other ways to solve the first two equations are equivalent to one of these by a renaming of the generators $v_k^{\pm}$\modif{, see the discussion around equation~\eqref{eq:otherv0}}.%
\begin{enumerate}
	\item[\ref{sssec:Iieven}] \textbf{Type I.i}: $1-N -\ell \equiv_m p$; $(\lambda + 1/2)^2 = -(\Lambda - i \kappa_0 \delta)^2$, and \modif{$(\lambda - N - 1/2)^2 = H_{\kappa}(1-N-\ell)$}.
	\item[\ref{sssec:Iiieven}]\textbf{Type I.ii}: $1-N -\ell \equiv_m 0$; $(\lambda + 1/2)^2 = -(\Lambda - i \kappa_0 \delta)^2$, and \modif{$(\lambda - N -  1/2)^2 = H_{\kappa}(1-N-\ell)$}.
	\item[\ref{sssec:IIeven}]\textbf{Type II}: $H_{\kappa}(1-N -\ell) =0$; $(\lambda + 1/2)^2 = -(\Lambda - i \kappa_0\delta)^2$, and $N-\lambda + 1/2 = 0$.
	\item[\ref{sssec:IIIeven}] \textbf{Type III}: $(\lambda +1/2)^2 = -(\Lambda - i\kappa_0 \delta)^2$, and $(\lambda - N - 1/2)^2 = -(\Lambda + (-1)^N i\kappa_0 \delta)^2$.
\end{enumerate}
The choice of these specific types is simply to normalize the expressions of $\Lambda$ as either 
\begin{align}
 \Lambda_1 &= i(\lambda + 1/2 +\kappa_0\delta), & &\text{or} & 	\Lambda_2 &= -i(\lambda +1/2 - \kappa_0\delta).
 \end{align}
 The two possibilities exist for $\Lambda$, but the representations are the same under the switch $(\Lambda_1, \delta) \to (-\Lambda_2,-\delta)$ so we will always only consider $\Lambda_1 = i(\lambda+1/2+\kappa_0\delta)$.

\subsubsection{Cases of type I.i} \label{sssec:Iieven}

The condition on $H_{\kappa}(1-N-\ell)$ is equivalent to $N+\ell \equiv_m 1-p$. 

There are two possible values for $\lambda$:
\begin{equation}
	\lambda_1 = N+1/2 + (\kappa_1 +\kappa_m)p, \qquad \lambda_2 = N+1/2 - (\kappa_1+\kappa_m)p.
\end{equation}

\paragraph{\emph{First option:} $\lambda_1 = N+1/2 + (\kappa_1+\kappa_m)p$}  All the eigenvectors $v_k^{\pm}$ have different eigenvalues. The condition $A(k)\neq 0$ is achieved with only some conditions on $\kappa_0$. Indeed, the  positivity conditions $\kappa_1, \kappa_m >0$ implies that $(\kappa_1 + \kappa_m)^2 > (\kappa_1 - \kappa_m)^2$ and thus the second factor $A_{(2)}(k)$ of $A(k)$ is always positive:
\begin{equation}
	A_{(2)}(k) = (\lambda_1 -k +1/2)^2 - H_{\kappa}(2-k-\ell) = (N-k+1 + (\kappa_1+\kappa_m)p)^2 - H_{\kappa}(2-k-\ell) >0.
\end{equation}
The first factor $A_{(1)}(k) = (\lambda_1 - k +1/2)^2 + (\Lambda + (-1)^{k+1} i \kappa_0 \delta)^2$ is always negative for even $k$, but it can be zero if $\kappa_0 = -k/(2\delta)$ or $\kappa_0 = -(2\lambda - k + 1)/(2\delta)$ for odd $k$.

When $\delta =1$, it will be unitary without restriction. When $\delta = -1$, it will be unitary if $\kappa_0 < 1/2$ or $\kappa_0 > \lambda_1$.

\paragraph{\emph{Second option:} $\lambda_2 = N+1/2 - (\kappa_1 + \kappa_m)p$.} As
\begin{align*}
	\sym{0} v_k^+ &= (N-k+/2 - (\kappa_1 + \kappa_m)p)v_k^+, & \sym{0} v_j^- &= j - N-1/2 + (\kappa_1 + \kappa_m)p,
\end{align*}
then $v_k^+$ and $v_j^-$ will have the same $\sym{0}$-eigenvalue when $\kappa_1 + \kappa_2 = (2N +1 - k -j)/2p$, values for which some $A(k)=0$.

In addition to the conditions on $\kappa_0$, some conditions on $\kappa_1$ and $\kappa_m$ appear from the irreducibility condition $A(k)\neq 0$. The factor $A_{(1)}(k)$ is  not zero as long as $\kappa_0 \neq -k/2\delta$ or $\kappa_0\neq -(N-k+1/2)/\delta$ for odd $k$, but it might be the case that the second factor $A_{(2)}(k)$ becomes $0$.

For $A_{(2)}(k)=0$, then the following equation must hold
\begin{align*}
	(\lambda_2 -k +1/2)^2 - H_{\kappa}(2-k-\ell) = (N-k +1 - (\kappa_1+\kappa_m)p)^2 - H_{\kappa}(2-k-\ell) = 0.
\end{align*}
According to the value of $k$, we get to solve
\begin{equation}
	\begin{cases}
	N-k+1 - (\kappa_1+\kappa_m)p = \pm(\kappa_1-\kappa_m)p, & 2-k-\ell \equiv_m 0;\\
	N-k+1 - (\kappa_1+\kappa_m)p = \pm(\kappa_1+\kappa_m)p, & 2-k-\ell \equiv_m p;\\
	N-k+1 - (\kappa_1+\kappa_m)p = 0, & \text{else};\\	
	\end{cases}
\end{equation}
and so $\modif{\stdquo}$ is not irreducible  when
\begin{equation}
	\begin{cases}
	\kappa_1 = (N-k+1)/(2p), \ \kappa_m = (N-k+1)/(2p), &  2-k-\ell \equiv_m 0; \\
	\kappa_1+\kappa_m = (N-k+1)/(2p), & 2-k-\ell \equiv_m p;\\
	\kappa_1 +\kappa_m = (N-k+1)/p, & \text{ else.}
	\end{cases}
\end{equation}

Lemma~\ref{lem:unitaritycond} states that the representation will be unitary when all the $A(k)>0$. If $\delta =1$, it suffices that $\kappa_0 >0$ and $\kappa_1 + \kappa_m < 1/p$. When $\delta =-1$, sufficient conditions for that are: $0<\kappa_0 < 1/2$ or $\kappa_0> N-1/2$ with $\kappa_1+\kappa_m < 1/p$. Indeed, $A_{(2)}(k) > 0$, and so $A(k) >0$, under the following restrictions:
\begin{align*}
	\begin{cases}
	\kappa_1, \kappa_m > (N-k+1)/2p \text{ or } \kappa_1,\kappa_m < (N-k+1)/2p, &2-k-\ell \equiv_m 0;\\
	\kappa_1+\kappa_m < (N-k+1)/2p, &   2-k-\ell \equiv_m p;\\
	\kappa_1 + \kappa_m < (N-k+1)/p,  &\text{else;}
	\end{cases}
\end{align*}
and as $2-N-\ell \not\equiv_m p,0$, then the condition $\kappa_1+\kappa_m < 1/p$ is sufficient for the inequality $A(k)>0$ to hold for all $k$.

\subsubsection{Cases of type I.ii}\label{sssec:Iiieven}
The condition on $H_{\kappa}(1-N-\ell) = (\kappa_1-\kappa_m)p$ is equivalent to $N+\ell \equiv_m 1$. There are two possibilities for $\lambda$, namely $\lambda_3 = N+1/2 + (\kappa_1 - \kappa_m)p$ or $\lambda_4 = N+1/2 - (\kappa_1 -\kappa_m)p$, and for each of these, we study $\Lambda_1 = i(\lambda_j + 1/2 + \kappa_0 \delta)$. We assume here that $\kappa_1>\kappa_m$ as, if it is not the case, it suffices to switch the analysis of $\lambda_3$ and $\lambda_4$.

\paragraph{\emph{First option:} $\lambda_3 = N+1/2 + (\kappa_1 - \kappa_m)p$ } There might be some $\sym{0}$-eigenvectors with equal eigenvalues as
\begin{align*}
	\sym{0} v_k^+ &=  N-k +1/2 + (\kappa_1 - \kappa_m)p, & \sym{0} v_j^- &= j-N-1/2 - (\kappa_1 - \kappa_m)p,
\end{align*}
and they are equal if $\kappa_m - \kappa_1 = (2N+1-k-j)/2p$.

The analysis on $A(k)$ proceeds similarly, but with a slight difference on the second factor of $A(k)$. Indeed, if $2-k-\ell \equiv_m p$ then
\begin{align*}
	A_{(2)}(k) = (\lambda_3 - k +1/2)^2 - H_{\kappa}(2-k-\ell) &= (N- k +1 + (\kappa_1-\kappa_m)p)^2 - (\kappa_1+\kappa_2)^2p^2
\end{align*}
and this is 0 if 
\begin{align*}
	\kappa_1 & = -(N-k+1)/2p, & \kappa_m &= (N - k +1)/2p.
\end{align*}
Then $A_{(2)}(k)$ will be positive if $\kappa_1 >0 > -(N-k+1)/2p$ and $0<\kappa_m < (N-k+1)/2p$. For the other values of $k$, then always $A_{(2)}(k) >0$. Naturally, $A_{(1)}(k) >0$ for even $k$, and is zero when $\kappa_0 = -k/2\delta$ or $\kappa_0 = -(\lambda_3 - k +1/2)/\delta$ for odd $k$.

When $\delta =1$, the condition for unitarity is then that $\kappa_m < (N-k+1)/2p$ for the biggest $k$ such that $2-k-\ell \equiv_m p$ and, if $\delta = -1$, that $\kappa_0 < 1/2$ or $\kappa_0 >\lambda_3$.

\paragraph{\emph{Second option:} $\lambda_4 = N+1/2 - (\kappa_1 -\kappa_m)p$}  The eigenvalues of $v_k^{\pm}$ might be the same as 
\begin{align*}
	\sym{0} v_k^+ &= (N+1/2 - k - (\kappa_1-\kappa_m)p)v_k^+ & \sym{0}v_j^- &= (j-N-1/2 + (\kappa_1-\kappa_m)p)v_j^-,
\end{align*}
and so they are the same if $\kappa_1 - \kappa_m = (2N-j-k+1)/2p$.

We get the following conditions on $\kappa_1$ and $\kappa_m$ to add to those on $\kappa_0$  by analysing when $A_{(2)}(k) =0$, or equivalently when
\begin{align*}
	(\lambda_4 - k +1/2)^2 = (N-k +1 - (\kappa_1-\kappa_m)p)^2
	&= \begin{cases}
	   	(\kappa_1+\kappa_m)^2p^2, & 2-k-\ell \equiv_m p;\\
	   	(\kappa_1-\kappa_m)^2p^2, & 2-k-\ell \equiv_m 0;\\
	   	0, & \text{else.}
	   \end{cases}
\end{align*}
It gives the following restrictions, for $k$ such that $2-k-\ell \equiv_m p$, $j$ such that $2-j-\ell \equiv_m 0$ and $q$ such that $2-q-\ell \not\equiv_m 0,p$
\begin{align*}
	\kappa_1 &\neq (N-k+1)/(2p), & \kappa_m &\neq -(N-k+1)/(2p),\\
	\kappa_1 - \kappa_m &\neq (N-j+1)/(2p), & \kappa_1-\kappa_m &\neq (N-q+1)/p.
\end{align*}

It is unitary if furthermore $A(k)>0$ for all $k$. This is achieved by $\kappa_1 -\kappa_2 < 1/p$ and $0<\kappa_1 < (N-k+1)/2p$ for the biggest $k$ such that $2-k-\ell \equiv_m 0$, when $\delta =1$, and by adding $0<\kappa_0 < 1/2$ or $\kappa_0 > \lambda_4 -k +1/2$, if $\delta = -1$. Note that the first condition ensures that $\lambda_4 >k$.

\subsubsection{Cases of type II}\label{sssec:IIeven}
This case results in 
\begin{align}
	\lambda_5 &= N+1/2, 
\end{align}
and we do the study for $\Lambda_1= i(\lambda_5+1/2+\kappa_0\delta)$. All the $v_k^{\pm}$ have different eigenvalues.

Similar analysis of $A(k)\neq 0$ gives $\kappa_0 \neq -k/2\delta$ and $\kappa_0\neq -(\lambda - k +1/2)/\delta$ for $k$ odd if the representation is to be irreducible. Furthermore, $A_{(2)}(k)$ might be zero. Indeed, $(\lambda - k +1/2)^2 =H_{\kappa}(2-k-\ell)$ when
\begin{equation*}
	\begin{cases}
    (N-k+1)^2 = (\kappa_1+\kappa_2)^2p^2, &  2-k-\ell \equiv_m p;\\
    (N-k+1)^2 = (\kappa_1-\kappa_2)^2p^2, &  2-k-\ell \equiv_m 0.\\
    \end{cases}
\end{equation*}
Or more precisely, we get the conditions
\begin{align}
	\begin{cases}
	\kappa_1+\kappa_2 = (N-k+1)/p, &  2-k-\ell \equiv_m p;\\
    |\kappa_1-\kappa_2| = (N-k+1)/p, &  2-k-\ell \equiv_m 0.\\
	\end{cases}
\end{align}
 The analysis proceeds in a similar fashion. The factor $A_{(2)}(k)>0$ as soon as $\kappa_1+\kappa_m < (N-k+1)/m$ for $2-k- \ell \equiv_m p$ or $|\kappa_1 - \kappa_m| < (N-k+1)/p$ for $2-k-\ell \equiv_m 0$, or without condition on $\kappa_1$ and $\kappa_m$ for other $k$. For $\delta =1$, then $A_{(1)}(k) <0$. For $\delta =-1$ then $A_{(1)}(k) <0$ for even $k$ and if $\kappa_0 < k/2$ or $\kappa_0 > \lambda_5+1/2 - k/2$ for odd $k$. Taking the minimum, or the maximum, of those set will ensure that $A(k)>0$ for all $k$.

\subsubsection{Cases of type III}\label{sssec:IIIeven}
We study $\Lambda_1 = i(\lambda +1/2+\kappa_0 \delta)$. The $\sym{0}$-eigenvalue $\lambda$ takes different values according to the parity of $N$. When $N$ is even, then
\begin{align}
	(\lambda - N -1/2)^2  + (\Lambda + i\kappa_0\delta)^2 =0,
\end{align}
so $\lambda_6 = N/2 + \kappa_0\delta$. When $N$ is odd, then $\lambda_7 = N/2$.

We begin by showing that $\lambda_7 = N/2$ for odd $N$ does not happen by the same argument as the $S_3$ case~\cite[Sect. 4.3.2]{debie_total_2018} and then continue with the case $\lambda_6 = N/2 + \kappa_0\delta$ for even $N$.

\paragraph{\emph{Odd} $N$ \emph{and} $\lambda_7 = N/2$} 
 If $\lambda_7 = N/2$, the integer $j_0 = (N-1)/2$ is such that
\begin{align}
	\sym{0} v_{j_0}^+ &= \tfrac{1}{2}v_{j_0}^+, & \sym{0} v_{j_0+1}^+ &= -\tfrac{1}{2} v_{j_0+1}^+, & \sym{0}v_{j_0}^- &= -\tfrac{1}{2} v_{j_0}^-, & \sym{0} v_{j_0+1}^- = \tfrac{1}{2} v_{j_0+1}^-.
\end{align}

We will show that the commutation relation involving $\comm{\sym{0},\sym{-}}$ is not respected, thus showing the impossibility of a representation. The actions of $\sym{+}$ and $\sym{-}$ previously found (equations~\eqref{eq:Omvpeven}--\eqref{eq:Opvmeven}) do not work here because their denominator is $0$. We can circumvent this by noticing that
\begin{align}
	v_{j_0+1}^+ &= \Lad{-}v_{j_0}^+ = (\sym{-}\sym{0} + \tfrac{1}{2}\comm{\sym{0},\sym{-}}) v_{j_0}^+\\
	&= \sym{-}(\sym{0} -\tfrac{1}{2})v_{j_0}^+ + 2(\sym{123} - T_0)T_-v_{j_0}^+\\
	&= i(\Lambda_1 - (-1)^{j_0}\kappa_0) G_{\kappa}(1 -j_0-\ell)v_{j_0}^-.
\end{align}

This forces $G_{\kappa}(1 -j_0-\ell) \neq 0$ as $v_{j_0+1}^+$ must not be zero. So $v_{j_0}^-$ is a multiple of $v_{j_0+1}^+$. But if this is the case, then studying the two sides of the commutator $\comm{\sym{0},\sym{-}}$ leads to a contradiction. Indeed, plugging in the last equation:
\begin{align}
	\comm{\sym{0},\sym{-}}v_{j_0}^+ &= -\sym{-} v_{j_0}^+ + 2(\sym{123} - T_0) T_-v_{j_0}^-\\
	\sym{0}\sym{-} v_{j_0}^+ - \tfrac{1}{2} \sym{-}v_{j_0}^+&= -\sym{-}v_{j_0}^+ +2v_{j_0+1}^+\\
	\sym{0}\sym{-}v_{j_0}^+ &= -\tfrac{1}{2} \sym{-}v_{j_0}^+ + 2v_{j_0+1}^+,
\end{align}
which is an impossible equation because $v_{j_0+1}^+ \neq 0$. There are thus no representations in this case.
 
\paragraph{\emph{Even} $N$ \emph{and} $\lambda_6 = N/2 + \kappa_0\delta$.} 
The $\sym{0}$-eigenvalues of $v_k^+$ and $v_j^-$ are
\begin{align}
	\sym{0} v_k^+ &= N/2 + \kappa_0 \delta - k, & \sym{0}v_j^{-} &= j-N/2 -\kappa_0\delta.
\end{align}
So they would be the same if $\kappa_0 = (k+j-N)/2\delta$. If $j$ and $k$ have the same parity, then they are distinguishable under the action of $\dsig{0}$. If they have the same parity, then the value of $\kappa_0$ is prohibited by the study of $A(k)\neq 0$.

Study $A(k)=0$. The first factor of $A(k) = - A_{(1)}(k)A_{(2)}(k)$ may be zero for some values of $\kappa_0$
\begin{equation}
	A_{(1)}(k) = (\lambda_6 - k +1/2)^2 - (\lambda_6 + 1/2 + \kappa_0\delta + (-1)^k\kappa_0\delta)^2 =0.
\end{equation}
For even $k$, then it forces $\kappa_0 = (k-N-1)/2\delta$. For odd $k$, it forces $\kappa_0 = -k/2\delta$ or $\kappa_0 = -(\lambda_6 - k/2+1/2)/\delta = -(N-k+1)/4\delta$.

The other factor can also cancel as
\begin{equation}
	A_{(2)}(k) = (\lambda_6 - k +1/2)^2 - H_{\kappa}(2-k-\ell) =0.
\end{equation}
For $2-k-\ell \equiv_m p$ then this happens if $\kappa_1+\kappa_m = \pm(\lambda_6 + k+1/2)/p$. For $2-k-\ell\equiv_m 0$ if $\kappa_1-\kappa_m = \pm(\lambda_6 - k +1/2)/p$. For the other $k$, this can also happen if $\kappa_0 = (N-2k+1)/2\delta$.

Unitarity is studied from the condition $A(k)>0$, for $1\leq k \leq N$, of Lemma~\ref{lem:unitaritycond}. The first factor $A_{(1)}(k)$ is negative under the assumption that 
\begin{equation}
	\begin{cases}
		|N/2 + \kappa_0\delta +1/2 - k| < |N/2 + \kappa_0\delta +1/2|, & \text{even }k;\\
		|N/2 + \kappa_0\delta + 1/2 - k| < |N/2 + 3\kappa_0\delta +1/2|, & \text{odd }k.
	\end{cases}
\end{equation}
 The first is then $\delta\kappa_0 > (k-N-1)/2$. The second divides according to $k$:
 \begin{equation}
 	\begin{cases}
 	\kappa_0\delta > -k/2 \text{ or } \kappa_0\delta < (k-N-1)/4, & k<(N+1)/3,\\
 	\kappa_0\delta < -k/2 \text{ or } \kappa_0\delta > (k-N-1)/4, & k>(N+1)/3.
 	\end{cases}
 \end{equation}
For $\delta =1$, this is always the case. If $\delta =-1$ then, for all $k$, it will require $\kappa_0 <1/2$ or $\kappa_0 > N/2+1/2$.

The second factor $A_{(2)}(k)$ is always positive when $k$ is such that $2-k-\ell \not\equiv_m 0,p$. When $2-k-\ell\equiv_m p$ then it is positive if $\kappa_1+\kappa_m < |\lambda_6 - k +1/2|/p$. And for $2-k-\ell\equiv_m 0$, then $|\kappa_1-\kappa_m| < |\lambda_6 - k +1/2|/p$.

The cases considered complete the proof of Theorem~\ref{thm:irrepseven}.\qed

\section{The monogenic representations}\label{sec:Mono}

This section contains a concrete realisation of a family of representations of the symmetry algebra $\SymAlg$. It consists of the  spaces of monogenics of the Dunkl--Dirac operator. They are built with the Fischer decomposition (Theorem~\ref{thm:fischer}) from the two-dimensional Dunkl--Laplace harmonics given by Dunkl~\cite{dunkl_differential-difference_1989,dunkl2014orthogonal} and takes the third dimension, with its $A_1$ root system contribution, into account via a Cauchy--Kovalevskaya extension (Theorem~\ref{thm:CK}).

Let $\Poly_n(\RR^N)$ denote the space of polynomials of degree $n$ in $N$ variables. Let $\Mono_n(\RR^3,\CC^2):= \ker \DDop \cap(\Poly_n(\RR^3)\otimes \CC^2)$ be the space of Dunkl monogenics of degree $n$ in $3$ variables. Here we make the identification of a spinor representation of the Clifford algebra $\Clif(3)$ with $\CC^2$ using Pauli matrices with an extra sign. Let $\delta\in\{-1,+1\}$. Realise the Clifford elements as one set of Pauli matrices: 
	\begin{align*}
	e_1 &\mapsto \begin{pmatrix}
	0&1 \\ 1&0
	\end{pmatrix}, 
	& e_2 &\mapsto \begin{pmatrix}
	0&-i\\i&0
	\end{pmatrix},
	& e_3 &\mapsto \begin{pmatrix}
	\delta&0\\0&-\delta
	\end{pmatrix}.
	\end{align*}
The difference given by $\delta$ is to account for the two two-dimensional irreducible representations of the Clifford algebra $\Clif(3)$. The pseudo-scalar $e_1e_2e_3$ acts as $i\delta$.

Using the Cauchy--Kovalevskaya extension and the Fischer decomposition theorem, we will construct a basis for the Dunkl monogenics $\Mono_n(\RR^3,\CC^2)$ of degree $n$ in 3 variables from the harmonics of the Dunkl--Laplace operator in two dimensions.

There is another realisation of $\osp(1|2)$ inside the algebra obtained by restricting to the $x_1$ and $x_2$ coordinates. Here we use the fact that the root system is reducible. Let the following denote the 2D counterparts of the operators defined in Section~\ref{sec:DDDeq}
\begin{align}
\widehat{\DDop} &:=  e_1 \Dun{1} + e_2 \Dun{2}, & \widehat{\xun} &:= e_1 x_1 + e_2x_2,\nonumber\\
\widehat{\DDop}^2 &= \Dun{1}^2 + \Dun{2}^2 =: \widehat{\DLapl}, & \widehat{\xun}^2 &= x_1^2 + x_2^2=:\widehat{\mathbf{x}}^2,\label{eq:osp12rel2d}\\
\widehat{\Euler} &:= x_1\partial_{x_1} + x_2\partial_{x_2}, & \widehat{\gamma} &:= {m\over 2}(\kappa_1 + \kappa_m). \nonumber
\end{align}
The operators $\widehat{\DDop}$ and $\widehat{\xun}$ respect
\begin{equation}\label{eq:commrelDDopxun2d}
\acomm{\widehat{\DDop},\widehat{\xun}} = 2(\widehat{\Euler} + 1 +\widehat{\gamma}),
\end{equation}
and they interact with the $\Lie{sl}_2$-triple $\widehat{\mathbf{x}}^2,\widehat{\DLapl},\widehat{\Euler}$ as follows: 
	\begin{equation}
	\begin{aligned}
	\comm{\widehat{\DDop},\widehat{\mathbf{x}}^2} &= 2\widehat{\xun}, & \comm{\widehat{\Euler},\widehat{\DDop}} &= -\widehat{\DDop},\\
	\comm{\widehat{\DLapl},\widehat{\xun}} &= 2\widehat{\DDop}, & \comm{\widehat{\Euler},\widehat{\xun}} &= \widehat{\xun}.
	\end{aligned}
	\end{equation}

 Recall some basic facts from hypergeometric analysis. The Pochhammer symbol $(a)_n$ is defined as $(a)_0 = 1$; $(a)_1 = a$ and $(a)_n = a_{n-1}(a+n-1) = a(a+1)\dots (a+n-1)$. The hypergeometric  series ${}_rF_s$ is given by
\begin{equation}
\hyperF{r}{s}{a_1,\dots , a_r}{b_1,\dots , b_s}{z} := \sum_{k=0}^{\infty} {(a_1)_k \dots (a_r)_k \over (b_1)_k \dots (b_s)_k} {z^k\over k!}.
\end{equation}

We will need some special orthogonal polynomials to present the results.
%
The \emph{Jacobi polynomial} of degree $n$ is given for constants $a$, $b$ as
\begin{equation}
P_n^{(a,\, b)} (x) := {(a +1)_n \over n!} \hyperF{2}{1}{-n,\, n+a+b+1}{a+1}{{1-x\over 2}}
\end{equation}
and they fulfil the identity
\begin{equation}\label{eq:identityJacobi}
(x+y)^n P_n^{(a,\, b)} \left( {x-y\over x+y}\right) = {(a+1)_n\over n!} x^n \hyperF{2}{1}{-n,\,-n-b}{a+1}{-{y\over x}}.
\end{equation}

Dunkl and Xu defined~\cite{dunkl2014orthogonal} the \emph{generalized Gegenbauer polynomials} by
\begin{equation}
\begin{aligned}\label{eq:Gegenbauer}
G_{2n}^{(\lambda,\,\mu)}(x) &= {(\lambda + \mu)_n \over (\mu + 1/2)_n} P_n^{(\lambda -1/2,\, \mu - 1/2)}(2x^2-1),\\
G_{2n+1}^{(\lambda,\,\mu)}(x) &= {(\lambda + \mu)_{n+1}\over (\mu + 1/2)_{n+1}} x P_n^{(\lambda -1/2,\, \mu+1/2)}(2x^2-1).
\end{aligned}
\end{equation}

The following proposition is extracted from the original paper of Dunkl~\cite{dunkl_differential-difference_1989} in the updated formulation of his and Xu's book~\cite{dunkl2014orthogonal} and gives a basis for the harmonics of the Dunkl-Laplacian.

\begin{proposition}[Dunkl, Sect. 3.14 and 3.19~\cite{dunkl_differential-difference_1989}]
	Let $n$ be a natural number and $D_{2m}$ be the dihedral group of order $2m$. There is a basis of the space of $D_{2m}$ Dunkl harmonics $\Harmo_n(\RR^2) := \ker \widehat{\DLapl} \cap \Poly_n(\RR^2)$ given by pairs of polynomials $\phi_n^{+}$, $\phi_n^{-}$ of degree $n$ depending on the parity of $m$. Denote $z := x_1+ix_2$ and $\conj{z} := x_1-ix_2$.
	\begin{itemize}
		\item (Odd $m$). Recall that then $\kappa_1=\kappa_m$. Decompose $n$ by Euclidean division as $n= km + \ell$ for $0\leq \ell <m$. 
		\begin{equation}
		\begin{aligned}
		\phi_n^+(x_1,x_2)  &= z^{\ell}\sum_{j=0}^n {(\kappa_1)_j(\kappa_1+1)_{n-j}\over j!(n-j)!} \conj{z}^{mj} z^{m(n-j)};\\
		\phi_n^-(x_1,x_2) &=\conj{\phi_n^+(x_1,x_2)}. 
		\end{aligned}
		\end{equation}
		\item (Even $m=2p$). Let $n = kp + \ell$ with $0 \leq \ell < p$. The harmonics polynomials are given by
		\begin{equation}
		\begin{aligned}
		\phi_n^+(x_1,x_2) &= z^{\ell} f_{k}(z^p,\conj{z}^p), \\
		\phi_n^-(x_1,x_2) &= \conj{z}^{\ell} f_k(\conj{z}^p,z^p),
		\end{aligned}
		\end{equation}
		with $f$ expressed with Gegenbauer polynomials and polar decomposition $z=re^{i\theta}$
		\begin{equation}
		f_k(z,\conj{z}) = r^k \left( {n+2\kappa_m + (1+(-1)^n)\kappa_1\over 2(\kappa_m+\kappa_1)} G_n^{(\kappa_m,\kappa_1)}(\cos(\theta)) + i \sin(\theta)G_{n-1}^{(\kappa_m+1,\kappa_1)}(\cos(\theta))\right).
		\end{equation}
	\end{itemize}
\end{proposition}

It is possible to rewrite $f_k(z,\conj{z})$ in a slightly more direct way for computational purposes~\cite{dunkl2014orthogonal}. Recall that $2x_1 =  (z+\conj{z})$ and $2ix_2 = z-\conj{z}$ and rewrite $f_k$ as
\begin{equation}
f_k(z,\conj{z}) = \begin{cases}
\dfrac{(\kappa_m+\kappa_1 + 1)_t} {(\kappa_1+1/2)_t} g_{2t}(z,\conj{z}), & k = 2t,\\ & \\
\dfrac{(\kappa_m+\kappa_1 + 1)_t}{(\kappa_1 + 1/2)_{t+1}} g_{2t+1}(z,\conj{z}), & k=2t+1,
\end{cases}
\end{equation}
with 
\begin{align*}
g_{2t}(z,\conj{z}) &= (-1)^t \sum_{j=0}^t {(-t+1/2-\kappa_m)_{t-j} (-t + 1/2 - \kappa_1)_j\over (t-j)!j!} x_1^{2t-2j}(ix_2)^{2j}\\
& \quad + (-1)^{t-1}\sum_{j=0}^{t-1} {(-t+1/2 - \kappa_m)_{t-1-j}(-t+1/2 - \kappa_1)_j\over (t-1-j)!j!}x_1^{2t-1-2j}(ix_2)^{(2j+1)},
\intertext{and}
g_{2t+1}(z,\conj{z}) &= (-1)^{t+1} \sum_{j=0}^t{(-t+1/2-\kappa_m)_{t+1-j} (-t + 1/2 - \kappa_1)_j\over (t-j)!j!} x_1^{2t+1-2j}(ix_2)^{2j}\\
& \quad + (-1)^{t+1}\sum_{j=0}^{t} {(-t+1/2 - \kappa_m)_{t-j}(-t+1/2 - \kappa_1)_{j+1}\over (t-j)!j!}x_1^{2t-2j}(ix_2)^{(2j+1)}.
\end{align*}

Knowing the Dunkl harmonics let us deduce the Dunkl monogenics. Let $\chi^+$ and $\chi^-$ be spinors  here expressed as the first and second coordinate vector $\chi^+ = (1,0)^T$ and $\chi^-=(0,1)^T$. The following proposition states a basis of Dunkl monogenics in two dimensions $\Mono_n(\RR^2,\CC^2) = \ker \DDop \cap(\Poly_n(\RR^2)\otimes \CC^2)$, where again we identify the spinor representation with $\CC^2$.

\begin{proposition}
	Let $n$ be a natural number. The polynomials
	\begin{align}
	\monopo^+_n(x_1,x_2) &= \phi_n^+(x_1,x_2) \chi^+ && \text{and} & \monopo_n^-(x_1,x_2) &= \phi_n^-(x_1,x_2)\chi^-
	\end{align}
	are a basis for the Dunkl monogenics $\Mono_n(\RR^2,\CC^2)$ of degree $n$.
\end{proposition}
\begin{proof}
	 Applying $\widehat{\DDop}$ on $\monopo^+_j$ in the Pauli matrices realisation yields
	\begin{align}
	\widehat{\DDop}\, \monopo^+_n = e_1 \Dun{1} \phi_n^+ \chi^+ + e_2\Dun{2}\phi^+_n \chi^+ = (\Dun{1} + i \Dun{2})\phi_n^+.
	\end{align}
	The polynomial $\phi_n^+$ is a harmonic of the Dunkl--Laplace operator. The Dunkl--Laplace operator factors as $\widehat{\Lapl} = (\Dun{1} + i\Dun{2})(\Dun{1} - i\Dun{2})$. Working out the properties of the function $f_k(z,\conj{z})$ shows that $\phi_n^+$ is annihilated by the first factor~\cite{dunkl_differential-difference_1989}, thus showing that $\Phi_n^+$ is a monogenic. Conjugating shows the result for the other polynomial $\Phi^-_n$.
\end{proof}

When the constants $\kappa_0$, $\kappa_1$ and $\kappa_m$ are positive, there is a decomposition of the space of spinor valued polynomials by Dunkl monogenics. 

\begin{theorem}[Fischer decomposition~\cite{orsted_howe_2009}]\label{thm:fischer}
	Let $n\in \NN$ and $\kappa_0,\kappa_1,\kappa_m >0$. There exists a decomposition of the space of spinor valued polynomials given by
	\begin{equation}
	\Poly_n(\RR^2)\otimes \CC^2 = \bigoplus_{j=0}^n \widehat{\xun}^{n-j} \Mono_j(\RR^2,\CC^2).
	\end{equation}	
\end{theorem}

Everything is in place for the Cauchy--Kovalevskaya extension Theorem. It establishes an isomorphism between the two-dimensional space and the three-dimensional monogenics taking into account the $\ZZ_2$ reflection group.

\begin{theorem}[Cauchy--Kovalevskaya, \cite{de_bie_diracdunkl_2016}]\label{thm:CK}
	Let $\kappa_0 >0$. There is an isomorphism between the spaces of spinor valued polynomials $\Poly_n(\RR^2)\otimes \CC^2$ and $\Mono_n(\RR^3,\CC^2)$ given on polynomials by
	\begin{equation}
	\CK_{x_3}^{\kappa_0} = \hyperF{0}{1}{-}{\kappa_0 + 1/2}{-{(x_3\widehat{\DDop})^2\over 4}} - {e_3 x_3 \widehat{\DDop}\over 2\kappa_0 +1} \hyperF{0}{1}{-}{\kappa_0 + 3/2}{-{(x_3\widehat{\DDop})^2\over 4}}. 
	\end{equation}
\end{theorem}

We are now ready to construct a basis of the monogenics. 

\begin{corollary}
	Let $n\in \NN$. A basis for the space $\Mono_n(\RR^3,\CC^2)$ is given by the $2n+2$ polynomials
	\begin{equation}
		\psi_{n,k}^{\pm}(x_1,x_2,x_3) := \CK_{x_3}^{\kappa_0}(\xun^{n-k} \Phi_k^{\pm}(x_1,x_2)), \quad k=0,\dots , n.
	\end{equation}
\end{corollary}

The polynomials $\psi_{n,k}^{\pm}$ can also be given explicitly. The next proposition works them out similarly to the $W = \ZZ_2\times\ZZ_2\times\ZZ_2$ case done in~\cite{de_bie_diracdunkl_2016}.
\begin{proposition}
	An explicit basis of the space of monogenics $\Mono_n(\RR^3,\CC^2)$ is given, for $k=0,\dots, n$, by
	\begin{equation}
	\begin{aligned}
	\psi_{n,k}^{+}(x_1,x_2,x_3) & =B_{n,k}(\widehat{\xun},x_3)\monopo_k^{+}(x_1,x_2), \\
	\psi_{n,k}^{-}(x_1,x_2,x_3) &= B_{n,k}(\widehat{\xun},x_3)\monopo_k^{-}(x_1,x_2),
	\end{aligned}
	\end{equation}
	with $B_{n,k}$ defined as
	\begin{equation}
	\begin{gathered}
	B_{n,k}(\widehat{\xun},x_3) =  \frac{t!}{(\kappa_0 + 1/2)_{t}}\xun^{2t} \times \\
	\begin{cases}	
	\left(\widehat{\xun} P_{t}^{(\kappa_0-1/2,\, k+1+\widehat{\gamma})}\left(\Upsilon(x)\right) - e_3 x_3\dfrac{t+k+1+\widehat{\gamma}}{ t+\kappa_0 + 1/2}P^{(\kappa_0 + 1/2,\, k+\widehat{\gamma})}_{t}\left( \Upsilon(x)\right)\right), & n-k = 2t+1;
	\\ \left( P_{t}^{(\kappa_0-1/2,\, k+\widehat{\gamma})}\left(\Upsilon(x)\right) - \dfrac{e_3x_3\widehat{\xun}}{\xun^2}P^{(\kappa_0 + 1/2,\, k+1+\widehat{\gamma})}_{t-1}\left(\Upsilon(x)\right)\right), & n-k = 2t;
	\end{cases}
	\end{gathered}
	\end{equation}
	and $\Upsilon(x) := (x_1^2+x_2^2-x_3^2)/(x_1^2+x_2^2+x_3^2)$.
\end{proposition}
\begin{proof}
	Let $M_k \in \Mono_k(\RR^2,\CC^2)$. The commutation relations~\eqref{eq:commrelDDopxun2d} and the fact that $M_k$ is a monogenic imply that 
	\begin{align}\label{eq:proofDxMk}
	\widehat{\DDop}^2( \widehat{\xun} M_k) &= 0, &
		\widehat{\DDop}(\widehat{\xun}^{2\beta +1} M_k ) &=  2(\beta + k + \widehat{\gamma}) \widehat{\xun}^{2\beta} M_k, & 
		\widehat{\DDop}(\widehat{\xun}^{2\beta} M_k)  &=  2\beta \widehat{\xun}^{2\beta-1} M_k. 
	\end{align}
From equations~\eqref{eq:proofDxMk}, a short computation generalizes to 
	\begin{equation}
	\widehat{\DDop}^a ( \widehat{\xun}^b M_k) = \begin{cases}
	d_{a,b}^k\,\widehat{\xun}^{\, b-a}M_k & a\leq b;\\
	0 & a> b;
	\end{cases}
	\end{equation}
	with the value of $d_{a,b}^k$ given by
	\begin{equation}
	d_{a,b}^k = \begin{cases}
	2^{2\alpha}(-\beta)_{\alpha}(-\beta - k -\widehat{\gamma})_{\alpha},        & a=2\alpha,\ b= 2\beta;\\
	-2^{2\alpha+1}(-\beta)_{\alpha+1}(-\beta-k-\widehat{\gamma})_{\alpha},      & a=2\alpha+1,\  b=2\beta;\\
	-2^{2\alpha+1}(-\beta)_{\alpha}(-\beta - k -1 -\widehat{\gamma})_{\alpha+1},& a=2\alpha+1,\  b=2\beta+1;\\
	2^{2\alpha}(-\beta)_{\alpha}(-\beta-k-1-\widehat{\gamma})_{\alpha},         & a=2\alpha,\  b=2\beta+1. 
	\end{cases}
	\end{equation}	
	It is now possible to use the anticommutation relation $\acomm{\widehat{\DDop},e_3} = 0$ and the identity~\eqref{eq:identityJacobi} of Jacobi polynomials to indeed obtain
	\begin{equation*}
	\psi_{n,k}^{\pm}(x_1,x_2,x_3) = B_{n-k}(\widehat{\xun},x_3)\monopo_k^{\pm}(x_1,x_2).
	\end{equation*}
\end{proof}

For $W = \ZZ_2\times D_{2m}$, there is an integral formulation of the inner product introduced abstractly in Section~\ref{sec:Repr}. Take the adapted weight function~\cite{dunkl2014orthogonal} with $z= x_1+ix_2$ and $\conj{z}=x_1-ix_2$
\begin{equation}\label{eq:unitaritymono}
	h_{\kappa_0,\kappa_1,\kappa_m}(x_1,x_2,x_3) := \left|{z^m + \conj{z}^m\over 2}\right|^{\kappa_1} \cdot \left|{z^m - \conj{z}^m\over 2i}\right|^{\kappa_m} \cdot |x_3|^{\kappa_0}.
\end{equation}
Let $X^{\dagger}$ be the transpose of $X$. Define an inner product by
\begin{equation}
	\bili{\psi_1}{\psi_2} := \int_{S_2} (\psi^{\dagger}_1 \psi_2) h^2_{\kappa_0,\kappa_1,\kappa_m}(x_1,x_2,x_3)\dd x_1\dd x_2\dd x_3.
\end{equation}

The structure of the monogenic representations is given in the next two propositions. 
\begin{proposition}\label{prop:monoirrepodd}
	Let $m = 2p+1$. For each $n\in \NN$, the space of monogenics $\Mono_n(\RR^3,\CC^2)$ of the Dunkl--Dirac operator of degree $n$ forms an irreducible representation of dimension $2n+2$ of the symmetry algebra $\SymAlg$ with basis 
	\begin{equation}
		\{\psi_{n,k}^{\pm} \mid k = 0,1\dots , n \}.
	\end{equation}
	The action of the symmetry algebra is given by
	\begin{align}
	\sym{0}\psi_{n,k}^{\pm} 	&= \pm (k + 1/2 + m\kappa_1)\psi_{n,k}^{\pm};&
	\sym{123}\psi_{n,k}^{\pm}	&= \delta i(n + 1 + \kappa_1m + \delta\kappa_0)\psi_{n,k}^{\pm},
	\end{align}
	where $\delta \in\{-1,+1\}$ comes from the realisation of the Clifford algebra element $e_3$. Let $k = rm + \ell$ with $0\leq \ell \leq m-1$ and $\zeta = e^{i\pi/m}$. The group $\dcover{W}$ action is given by 
		\begin{equation}
	\begin{aligned}
	\dsig0 \psi_{n,k}^{\pm} &= \pm\delta(-1)^{n-k} \psi_{n,k}^{\pm};&
	\dsig1 \psi_{n,k}^{\pm} &= \mp i (-1)^{(n-k)} \zeta^{\pm 2\ell}\zeta^{\pm 1} \psi_{n,k}^{\mp};\\
	\dsig m \psi_{n,k}^{\pm} &= \pm i (-1)^{n-k}\psi_{n,k}^{\mp};&
	\dtau \psi_{n,k}^{\pm} &= -\zeta^{\mp(2\ell+1)}\psi_{n,k}^{\pm}.
	\end{aligned}
	\end{equation}
    Furthermore, the representation is unitary.
\end{proposition}

\begin{proof}
	Recall $\sym{123} = \tfrac{1}{2}(\comm{\DDop,\, \xun} -1)e_1e_2e_3 $, see equation~\eqref{eq:O123ori}. On any monogenic $\psi_{n,k}$ of degree $n$ we have
	\begin{align*}
	 1/2(\comm{\DDop,\xun}-1)\psi_{n,k} &= 1/2(\DDop\,\xun -1)\psi_{n,k}\\
	&= 1/2(\acomm{\DDop,\xun} -1)\psi_{n,k} = 1/2(2\Euler+3+2\gamma-1)\psi_{n,k}= (n+1+\kappa_1m+\delta\kappa_0)\psi_{n,k}
	\end{align*}
	and thus, according to the realisation of $e_3$,
	\begin{align}
	\sym{123}\psi_{n,k}=\delta i(n+1+\kappa_1m+\delta\kappa_0)\psi_{n,k}^{\pm}.
	\end{align}
	Restricting to the plane $x_1$, $x_2$, we have that $\sym{0} = -\tfrac{i}{2}e_1e_2(\comm{\widehat{\DDop},\widehat{\xun}}-1)$ is the Scasimir of the $\osp(1|2)$ realisation~\eqref{eq:osp12rel2d} times the pseudo-scalar $e_1e_2$. On 2D monogenics, we thus have
	\begin{align*}
	(\comm{\widehat{\DDop},\widehat{\xun}}-1)\Phi_k^{\pm} &= 1/2(\acomm{\widehat{\DDop},\widehat{\xun}} -1)\Phi_k^{\pm}= 1/2(2\Euler + 2 + 2m\kappa_1 - 1) \Phi_k^{\pm} = (k+1/2+m\kappa_1)\Phi_{k}^{\pm}.
	\end{align*}
	With $-ie_1e_2\chi^{\pm} = \pm \chi^{\pm}$, we get
	\begin{equation}
	\sym{0} \psi_{n,k}^{\pm} = \pm (k+1/2 + m\kappa_1)\psi_{n,k}^{\pm}.
	\end{equation}
	By direct computations on the explicit expression of $\phi_k^{\pm}$ we have 
	\begin{align}
	\sigma_0 \phi_{k}^{\pm} &= \phi_k^{\pm},&
	\sigma_m \phi_k^{\pm} &= \overline{\phi_k^{\pm}} = \phi_k^{\mp},&
	\sigma_1 \phi_k^{\pm} &= \zeta^{\pm 2\ell}\phi_k^{\mp}.
	\end{align}
	By example, for $\sigma_1$, note that $\sigma_1(z) = e^{2\pi i /m}\conj{z}$ and  $\sigma_1(\conj{z}) = e^{-2\pi i /m}z$, so
		$$
			\sigma_1 \phi_k^+ = \sigma_1 (z^{\ell}) C_k^{(\kappa,\kappa+1)}(z^m,\conj{z}^m) = \zeta^{2\ell}\conj{z}^{\ell} C_k^{(\kappa,\kappa+1)}(\zeta^{2m} \conj{z},\zeta^{-2m}z^m) = \zeta^{2\ell}\phi^-_k.
		$$
	Furthermore
	\begin{align*}
	\dsig{1} \Phi_k^{\pm} &= \sigma_1(\sin(\pi/m)e_1-\cos(\pi/m)e_2)\phi_k^{\pm}\chi^{\pm}\\
	&= \zeta^{\pm 2\ell} \phi^{\mp}_k (\sin(\pi/m) \mp i\cos(\pi/m))\chi^{\mp} = \mp i \zeta^{\pm 2\ell} \zeta^{\pm 1} \Phi_k^{\mp}.
	\end{align*}

	Adding $e_1 \chi^{\pm} = \chi^{\mp}$, $e_2\chi^{\pm} = \pm i \chi^{\mp}$ and $e_3\chi^{\pm} = \delta \chi^{\pm}$, and the fact that both $\dsig{1}$ and $\dsig{m}$ anticommute with $\widehat{\xun}$, we get what is needed.
	
	To prove that it is irreducible, it is sufficient to prove that each $\psi_{n,k}^{\pm}$ generates the whole representation. From the previous computations, the representation is a renormalized version of the irreducible representation constructed in the no-restriction subcases of cases I of Theorem~\ref{thm:irrepsodd}	with a switch from $v_k^{\pm}$ to $C(n,k)\psi_{n,n-k}^{\mp}$ for certain non-zero constants $C(n,k)$. Therefore, acting on $\psi_{n,k}^{\pm}$ with the operators $\Lad{\pm}$ will be enough to travel between indices of $\psi_{n,k}^{\pm}$.
	
	Unitarity comes from the definition of the weight function~\eqref{eq:unitaritymono}, the Dunkl harmonics used to construct the monogenics and from the case I of Theorem~\ref{thm:irrepsodd}.
 \end{proof}

 A similar proposition holds when $m$ is even. This representation is a renormalized version of the no-restriction subcase of cases I.i in Theorem~\ref{thm:irrepseven} when sending $v_k^{\pm}$ to $C'(n,k)\psi_{n,n-k}^{\mp}$ for some non-zero constants $C'(n,k)$.

\begin{proposition}\label{prop:monoirrepeven}
	Let $m = 2p$. For each $n\in \NN$, the space of monogenics $M_n(\RR^3,\CC^2)$ of the Dunkl--Dirac operator of degree $n$ form an irreducible and unitary representation of dimension $2n+2$ of the symmetry algebra $\SymAlg$ with basis 
	\begin{equation}
		\{\psi_{n,k}^{\pm} \mid k = 0,1\dots , n \}.
	\end{equation}
	The action of the symmetry algebra is given by
	\begin{align}
	\sym{0}\psi_{n,k}^{\pm} 	&= \pm (k + 1/2 + p(\kappa_1+\kappa_m))\psi_{n,k}^{\pm};&	\sym{123}\psi_{n,k}^{\pm}	&= \delta i(n + 1 + p(\kappa_1+\kappa_m)+\delta\kappa_0)\psi_{n,k}^{\pm},
	\end{align}
	where $\delta \in\{-1,+1\}$ comes from the realisation of the Clifford algebra element $e_3$. Let $k = rm + \ell$ with $0\leq \ell \leq m-1$ and $\zeta = e^{i\pi/m}$. The group $\dcover{W}$ action is given by 
		\begin{equation}
	\begin{aligned}
	\dsig0 \psi_{n,k}^{\pm} &= \pm\delta (-1)^{n-k} \psi_{n,k}^{\pm};&
	\dsig1 \psi_{n,k}^{\pm} &= \mp i (-1)^{(n-k)} \zeta^{\pm 2\ell}\zeta^{\pm 1} \psi_{n,k}^{\mp};\\
	\dsig m \psi_{n,k}^{\pm} &= \pm i (-1)^{n-k}\psi_{n,k}^{\mp};&
	\dtau \psi_{n,k}^{\pm} &= -\zeta^{\mp(2\ell+1)}\psi_{n,k}^{\pm}.
	\end{aligned}
	\end{equation}
\end{proposition}

%
%

\section{Concluding remarks}\label{sec:conclusion}
We here recall the main results and present the scope of the methods used. On the general 3D symmetry algebra of the Dunkl--Dirac operator, we added Proposition~\ref{prop:sqO123} giving the square of the symmetry $\sym{123}$ for any root system. Specifically for all reducible root systems of rank $3$, we listed all the finite-dimensional irreducible representations as well as sufficient conditions for their unitarity. A polynomial family of irreducible and unitary representations was realised through the important  example of monogenics.

The idea employed here will be difficult to apply to other higher-rank root systems as the ladder operators trick will be likely to fail. However, as it covers all the rank 2 cases, it can serve as a base case on which to support the jump for higher dimensions. Another point of this article was to work out completely the details given from adding a $\ZZ_2$ direct product to the reflection group. 

In future work, we are planning to use the insight gained here to study the $S_n$ reflection group. The hope would be that the rank 3 root system $A_3$ linked to $S_4$ will prove sufficient to be the base case of the study.

Another direction we see stemming from this work would be to formalize the argument on the impact of adding $\ZZ_2$ as direct product to any reflection group. This addition changes the double coverings to a central semi-direct product (see Gorenstein's book~\cite{gorenstein_finite_2007}). We saw that the $\ZZ_2$ part acted as a ``glue'' on the irreducible representations. 

We gave one important realisation of the irreducible representation in Section~\ref{sec:Mono}. However, there are many more possible families of irreducible representations available, as readily seen from Theorems~\ref{thm:irrepsodd} and~\ref{thm:irrepseven}, but we do not know of concrete useful examples of them. As the values where irreducibility and unitarity fail resemble conditions appearing in related work, for example in~\cite{chmutova_representations_2006,de_bie_lian_2020}, it seems interesting to link them. It would also serve as a motivation to study the structure of the representations when they are reducible.

%
%

\section*{Acknowledgements}
We wish to thank the anonymous referee for their careful reading of the manuscript and useful remarks. This project was supported in part by the EOS Research Project [grant number 30889451]. Moreover, ALR holds a scholarship from the Fonds de recherche du Qu\'ebec -- Nature et technologies, [grant number 270527] and RO was supported by a postdoctoral fellowship, fundamental research, of the Research Foundation -- Flanders (FWO), [grant number 12Z9920N].

\begingroup 
\raggedright
\singlespacing
\bibliographystyle{abbrv}
\bibliography{DBLROVdJ_dihedral_revised}
\endgroup

\appendix


\section{Double coverings}\label{app:doublecover}

In this section, the results concerning double covering groups and representation theory are recalled. The main source for this material is the important work of Morris~\cite{Morris76}.

The double coverings of reflection groups come from the restriction of the two non-trivial double coverings of the orthogonal group $O(n)$. For a refreshment on Clifford theory, see~\cite{karoubi1968} or~\cite{atiyah1964}. 

\begin{theorem}[Atiyah, Bott and Shapiro~\cite{atiyah1964}]\label{thm:pin}
The sequences
\begin{equation}
\begin{tikzcd}
0 \arrow[r]& \mathbb Z_2 \arrow[r]& \Pin_+(n) \arrow[r] & \Ortho(n) \arrow[r] & 0
\end{tikzcd}
\end{equation}
\begin{equation}
\begin{tikzcd}
0 \arrow[r]& \mathbb Z_2 \arrow[r]& \Pin_-(n) \arrow[r] & \Ortho(n) \arrow[r] & 0
\end{tikzcd}
\end{equation}
are exact and $\Pin_+(n)$ and $\Pin_-(n)$ are in general two non-isomorphic double covering of $\Ortho(n)$.
\end{theorem}

Let $W$ be a Coxeter group with Coxeter presentation given by some $m_{ij}\in \NN$
\begin{equation}
W= \left\langle \sigma_1, \dots , \sigma_n \ \middle| \ (\sigma_i\sigma_j)^{m_{ij}} = 1;\, m_{ii} =1\right\rangle.
\end{equation}

From Theorem~\ref{thm:pin}, one gets in general two double coverings of the Coxeter group $W$.
\begin{corollary}[{\cite[Prop. 3.5]{Morris76}}]
The two sequences
\begin{equation}
\begin{tikzcd}
0 \arrow[r]& \mathbb Z_2 \arrow[r]& \dcover{W}^+ \arrow[r] & W \arrow[r] & 0
\end{tikzcd}
\end{equation}
\begin{equation}
\begin{tikzcd}
0 \arrow[r]& \mathbb Z_2 \arrow[r]& \dcover{W}^- \arrow[r] & W \arrow[r] & 0
\end{tikzcd}
\end{equation}
are exact and $\dcover{W}^+$ and $\dcover{W}^-$ are central extensions of $W$.
\end{corollary}

In fact, there are even generator and relations presentations for $\dcover{W}^+$ and $\dcover{W}^-$.

\begin{theorem}[{Morris~\cite[Thm 3.6]{Morris76}}]\label{thm:morrisdc}
	The two double coverings of $W$ are given by
	\begin{align}
	\dcover{W}^+ &= \left\langle z, \dsig{1}, \dots , \dsig{n} \ \middle| \ z^2 =1, \begin{matrix} (\dsig i\dsig j)^{m_{ij}} = 1, & m_{ij} \text{ odd,}\\  (\dsig{i}\dsig{j})^{m_{ij}} = z, & m_{ij}\text{ even} \end{matrix}\right\rangle,\\
	\dcover{W}^- &=	\left\langle z,\dsig 1, \dots , \dsig n \ \middle| \ z^2=1, (\dsig i\dsig j)^{m_{ij}} = z\right\rangle.
	\end{align}
\end{theorem}
Remark that the proposition does not entail that the two double coverings are different, and it might be the case that $\dcover{W}^+$ is not a central extension of the group $W$, for example when $W= D_{2m}$ for $m$ odd~\cite{Schur1907}. Furthermore, we here took the liberty to replace the usual notation $-1 \in \dcover{W}$ by $z$ following Humphreys's and Hoffman's example~\cite{hoffman_projective_1992} to avoid confusion.

When considering representations of $\dcover{W}^-$ and $\dcover{W}^+$, those where the commuting element $z$ is acting as the identity are in correspondence with the representations of $W$. Those where $z$ acts as $-1$ are in correspondence with projective representations of $W$ and are called \emph{spin representations}~\cite{Morris76}.

The group considered in this article is $W= \ZZ_2 \times D_{2m}$ and it is presented by
\begin{equation}
W=\left\langle \sigma_0,\sigma_1,\sigma_m \ \middle| \ \sigma_0^2=\sigma_1^2=\sigma_m^2 = (\sigma_0\sigma_1)^2 = (\sigma_0\sigma_m)^2= (\sigma_1\sigma_m)^m =1 \right\rangle,
\end{equation}
 where $\sigma_0$ is the generator for $\ZZ_2$. We have presentations by generators and relations from Theorem~\ref{thm:morrisdc}.
	
 \begin{corollary}\label{coro:doublecoverings}
	The two central extensions of $W$ have the following presentations by generators and relations depending on the parity of $m$.
	\begin{itemize}
		\item ($m$ odd)
		\begin{equation}
		\dcover{W}^+ = \left\langle z, \dsig 0,\dsig 1,\dsig m \ \middle| \ z^2 = \dsig 0^2=\dsig 1^2=\dsig m^2 = (\dsig 1\dsig m)^m = 1;\; (\dsig 0\dsig 1)^2 = (\dsig 0\dsig m)^2= z \right\rangle,
		\end{equation}
		\begin{equation}
		\dcover{W}^- = \left\langle z,\dsig 0,\dsig 1,\dsig m \ \middle| \ z^2 = 1;\; \dsig 0^2=\dsig 1^2=\dsig m^2  = (\dsig 0\dsig 1)^2 = (\dsig 0\dsig m)^2= (\dsig 1\dsig m)^m= z \right\rangle.
		\end{equation}
			\item ($m$ even)
		\begin{equation}
		\dcover{W}^+ = \left\langle z,\dsig 0,\dsig 1,\dsig m \ \middle| \ z^2 = \dsig 0^2=\dsig 1^2=\dsig m^2  = 1;\; (\dsig 0\dsig 1)^2 = (\dsig 0\dsig m)^2= (\dsig 1\dsig m)^m= z \right\rangle,
		\end{equation}
		\begin{equation}
		\dcover{W}^- = \left\langle z,\dsig 0,\dsig 1,\dsig m \ \middle| \ z^2 = 1;\; \dsig 0^2=\dsig 1^2=\dsig m^2  = (\dsig 0\dsig 1)^2 = (\dsig 0\dsig m)^2= (\dsig 1\dsig m)^m= z \right\rangle.
		\end{equation}
	\end{itemize}
\end{corollary}

From this corollary, we can construct all the finite-dimensional irreducible representations for $\dcover{W}^+$ and $\dcover{W}^-$ in the odd and even cases. %

The classical idea followed here is simply to give the conjugacy classes and then construct as many non-equivalent irreducible finite-dimensional representations, thus exhibiting them all. The results are summarised in Theorem~\ref{thm:irrepsW} at the end of the appendix. We included all the details for $\dcover{W}^+$ as this material is hard to find in recent literature and seems to us of good pedagogical value. 
  
 \subsection{Irreducible representations for the odd case}\label{app:doublecoverodd}
 When $m = 2p+1$ is odd, there are $4p+5 = 2m+3$ conjugacy classes for $\dcover{W}^+$. For ease of notation, let $\dtau := \dcover{\sigma}_1\dcover{\sigma}_m$ be the even element of order $m$. We start by listing the conjugacy classes of $\dcover{W}^+$:
\begin{equation}	
 \begin{gathered}
 	\{1\}, \quad \{z\}, \quad \{\dsig{0},z\dsig{0}\},\\
 	\{\dtau,\dtau^{2p}\}, \{\dtau^2,\dtau^{2p-1}\} , \dots , \{\dtau^p,\dtau^{p+1}\}, \\
 	\{z\dtau,z\dtau^{2p}\}, \{z\dtau^2,S\dtau^{2p-1}\} , \dots , \{z\dtau^p,z\dtau^{p+1}\},\\
	\{\dsig{0}\dtau,z\dsig{0}\dtau^{2p}\}, \{\dsig{0}\dtau^2,z\dsig{0}\dtau^{2p-1}\} , \dots , \{\dsig{0}\dtau^p,z\dsig{0}\dtau^{p+1}\},\\
	\{z\dsig{0}\dtau,\dsig{0}\dtau^{2p}\}, \{z\dsig{0}\dtau^2,\dsig{0}\dtau^{2p-1}\} , \dots , \{z\dsig{0}\dtau^p,\dsig{0}\dtau^{p+1}\},\\
	\{\dsig{m}, \dtau\dsig{m} , \dtau^2\dsig{m}, \dots , \dtau^{2p}\dsig{m}, z\dsig{m}, z\dtau\dsig{m} , \dots , z\dtau^{2p}\dsig{m}\},\\
	\{\dsig{0}\dsig{m}, \dsig{0}\dtau\dsig{m} , \dsig{0}\dtau^2\dsig{m}, \dots , \dsig{0}\dtau^{2p}\dsig{m}, z\dsig{0}\dsig{m}, z\dsig{0}\dtau\dsig{m} , \dots , z\dsig{0}\dtau^{2p}\dsig{m}\}.
 \end{gathered}
 \end{equation}
 And indeed, counting the elements in the conjugacy classes gives:
 \begin{equation*}
 1 + 1 + 2 + p\times 2 + p\times 2 + p\times 2 + p\times 2 + (4p+2) + (4p+2) = 16p+8 =8m =|\dcover{W}^+|.
 \end{equation*}
 
We now construct the $4p+5$ non-equivalent irreducible finite-dimensional representations. 

Let $V$ be a finite-dimensional vector space and $X: \dcover{W}^+ \longrightarrow Gl(V)$  be a non-trivial finite-dimensional irreducible representation of $\dcover{W}^+$. Consider an eigenvector for $\dtau$, $z$ and $\dsig{0}$, noted $v_1\in V$. As $z^2 = \dsig{0}^2 = \dtau^m = 1$, we have that
 \begin{align}
 zv_1 &= \varepsilon v_1, & \varepsilon \in \{-1,+1\};\\
 \dsig{0}v_1 &= \delta v_1, & \delta \in \{-1,+1\};\\
 \dtau v_1 &= \zeta^{2\ell}v_1, & \zeta := e^{\pi i / m}, \ \ell\in \{0,1,\dots, m-1\}.
 \end{align}
 Put $v_2 := \dsig{m} v_1$. It follows from $\dtau \dsig{m}\dtau=\dsig{m}$ that
 \begin{align*}
 \dtau v_2 &= \dtau\dsig{m}v_1 = \dsig{m}\dtau^{-1}v_1 = \zeta^{-2\ell} v_2,&
 \dsig{m}v_2 &= \dsig{m}^2v_1 = v_1,\\
 zv_2 &= \dsig{m}zv_1 = -\dsig{m}v_1 = -v_2,&
 \dsig{0} v_2 &= \dsig{0}\dsig{m}v_1 = z\dsig{m}\dsig{0}v_1 = \varepsilon\delta v_2.
 \end{align*}
 This means that $\langle v_1,v_2 \rangle$ is a submodule of $V$; as $V$ is irreducible, $\langle v_1,v_2 \rangle = V$ and thus $\dim V \leq 2$.
 
The process divides in two cases whether $\zeta^{2\ell} = \zeta^{-2\ell}$ or not. First, assume $\zeta^{2\ell} \neq \zeta^{-2\ell}$. Then $\ell\in \{1,\dots, m-1\}$. In this case, $v_1$ and $v_2$ have two different eigenvalues, and they are therefore linearly independent, so the dimension of $V$ is $2$. 
 
 The matrices of the representation $X$ in the basis $\{v_1,v_2\}$ are given by
\begin{equation}
 \begin{aligned}
	X(\dtau) &= \begin{pmatrix}
	\zeta^{2\ell} & 0 \\ 0 & \zeta^{-2\ell}
	\end{pmatrix}, & X(\dsig{m}) &= \begin{pmatrix}
	0&1\\1&0
	\end{pmatrix}, &
	X(z) &= 
	\begin{pmatrix}
	\varepsilon & 0 \\ 0&\varepsilon	
	\end{pmatrix}, & X(\dsig{0}) &= \begin{pmatrix}
	\delta & 0 \\ 0&\varepsilon\delta
	\end{pmatrix}.
 \end{aligned}
 \end{equation}
 
 Now count how many non-equivalent representations this gives. A first remark is that we can ask of the imaginary part of $\zeta^{2\ell}$ to be positive. Indeed, if it is not the case then for $\varepsilon = 1$, switching $v_1$ and $v_2$ will make it so, and for $\varepsilon = -1$, changing $v_1$ and $v_2$ and sending $\delta$ to $-\delta$ will give an equivalent representation. With the positive condition on the imaginary part of $\zeta^{2\ell}$, such repetitions are avoided. This condition results in a restriction on the values $\ell$ can take: $\ell\in \{1,\dots , p\}$. The values of $\delta$ and $\varepsilon$ are independent in the set $\{-1,+1\}$ and so there are a total of $4p$ non-equivalent irreducible representations of dimension 2.
 
 Second, assume $\zeta^{2\ell} = \zeta^{-2\ell}$. As $m = 2p+1$, this only let $\zeta^{2\ell} = 1$, so $\ell=0$. Then the matrices in the generating set $\{v_1,v_2\}$ are given by
 \begin{equation}
 \begin{aligned}
 X(\dtau) &= \begin{pmatrix}
 1 & 0 \\ 0 & 1
 \end{pmatrix}, & X(\dsig{m}) &= \begin{pmatrix}
 0&1\\1&0
 \end{pmatrix},&
 X(z) &= 
 \begin{pmatrix}
 \varepsilon & 0 \\ 0&\varepsilon	
 \end{pmatrix}, & X(\dsig{0}) &= \begin{pmatrix}
 \delta & 0 \\ 0&\varepsilon\delta
 \end{pmatrix}.
 \end{aligned}
 \end{equation}
 
Further divide according to the value of $\varepsilon$. If $\varepsilon =1$, then notice that the actions on $v_1+v_2$ and $v_1-v_2$ are given by
\begin{align*}
\dtau(v_1+v_2) &=  v_1+v_2, & \dsig{m} (v_1+v_2) &= v_2+v_1, &
z(v_1+v_2) &= v_1+v_2, & \dsig{0}(v_1+v_2) &= \delta(v_1+v_2);\\
\dtau(v_1-v_2) &=  v_1-v_2, & \dsig{m} (v_1-v_2) &= -(v_1-v_2),&
z(v_1-v_2) &= v_1-v_2, & \dsig{0}(v_1-v_2) &= \delta(v_1-v_2).
\end{align*}
 Therefore, both $\langle v_1 + v_2\rangle$ and $\langle v_1-v_2\rangle$ are submodules of $V$. The irreducibility of $V$ forces one of them to be trivial and the other, to generate $V$. So $v_2 = v_1$ or $v_2 = -v_1$. Adding the choice of value for $\delta$, this gives $4$ one-dimensional irreducible representations.
 
When $\varepsilon  = -1$, the vectors $v_1$ and $v_2$ have a different eigenvalue for $\dsig{0}$, so they are linearly independent, and thus $v_1$ and $v_2$ are a basis for a two-dimensional representation. The two representations given by $\delta = -1$ and $\delta = 1$ are equivalent after switching $v_1$ and $v_2$ so there is only one more two-dimensional irreducible representation.

The total is $4p + 5$ irreducible representations, the same number as conjugacy classes, so all of them have been found. The dimensions match as indeed the sum of the squares of the dimensions gives the order of the group: $4p \times 2^2 + 4\times 1^2 + 1 \times 2^2 = 16p + 8 = |\dcover{W}^+|$.
 
Similar steps will also give the $4p+5$ irreducible representations of $\dcover{W}^-$. The only differences are that the extra relations constrict the values of $\varepsilon$ according to the action of $\dtau$ and that $\dsig{m}$ possibly is of order $4$, so $\delta$ takes values in $\{-1,+1,-i,+i\}$. All the representations are given in Theorem~\ref{thm:irrepsW}.

 \subsection{Irreducible representations for the even case} \label{app:doublecovereven}
 When $m=2p$ is even, there are $4p + 6 = 2m+6$ conjugacy classes for $\dcover{W}^+$. Still keeping the shorthand notation $\dtau := \dsig{1}\dsig{m}$ they go as
 \begin{equation}
 \begin{gathered}
  	\{1\}, \quad \{z\}, \quad \{\sigma_0,z\sigma_0\},\\
 \{\dtau,z\dtau^{2p-1}\}, \{\dtau^2,z\dtau^{2p-2}\} , \dots , \{\dtau^p,z\dtau^{p}\}, \{\dtau^{p+1},z\dtau^{p-1}\} \dots, \{\dtau^{2p-1}, -\dtau\},\\
 \{\dsig{0}\dtau,\dsig{0}\dtau^{2p-1}\}, \{\dsig{0}\dtau^2,\dsig{0}\dtau^{2p-2}\} , \dots , \{\dsig{0}\dtau^{p-1},\dsig{0}\dtau^{p+1}\}, \{\dsig{0}\dtau^p\},\\
 \{z\dsig{0}\dtau,z\dsig{0}\dtau^{2p-1}\}, \{z\dsig{0}\dtau^2,z\dsig{0}\dtau^{2p-2}\} , \dots , \{z\dsig{0}\dtau^{p-1},z\dsig{0}\dtau^{p+1}\}, \{z\dsig{0}\dtau^p\},\\
 \{\dsig{m}, \dtau^2\dsig{m}, \dots , \dtau^{2p-2}\dsig{m}, z\dsig{m}, z\dtau^2\dsig{m} , \dots , z\dtau^{2p-2}\dsig{m}\},\\
 \{\dsig{0}\dsig{m}, \dsig{0}\dtau^2\dsig{m} , \dots , \dsig{0}\dtau^{2p-2}\dsig{m}, - \dsig{0}\dsig{m}, z\dsig{0}\dtau^2\dsig{m} , \dots , z\dsig{0}\dtau^{2p-2}\dsig{m}\},\\
  \{\dtau\dsig{m}, \dtau^3\dsig{m}, \dots , \dtau^{2p-1}\dsig{m}, z\dtau\dsig{m}, z\dtau^3\dsig{m} , \dots , z\dtau^{2p-1}\dsig{m}\},\\
 \{\dsig{0}\dtau\dsig{m}, \dsig{0}\dtau^3\dsig{m} , \dots , \dsig{0}\dtau^{2p-1}\dsig{m}, - \dsig{0}\dtau\dsig{m}, z\dsig{0}\dtau^3\dsig{m} , \dots , z\dsig{0}\dtau^{2p-1}\dsig{m}\}.
 \end{gathered}
 \end{equation}
Adding the elements of the conjugacy classes gives the order of the group $\dcover{W}^+$:
\begin{equation*}
1+1+2+(2p-1)\times 2 + (p-1)\times 2 + 1 + (p-1)\times 2 + 1 + 2p + 2p +2p +2p = 16p = 8m = |\dcover{W}^+|.
\end{equation*}
Let $V$ be a finite-dimensional vector space and $X: \dcover{W}^+ \longrightarrow GL(V)$ be a non-trivial finite-dimensional irreducible representation of $\dcover{W}^+$. Take $v_1\in V$ to be an eigenvector for $\dtau$, $z$ and $\dsig{0}$. From $\dtau^{m} = z$, $z^2= 1$ and $\dsig{0}^2 =1$, it  follows 
\begin{align}
zv_1 &= \varepsilon v_1, & \varepsilon \in \{-1,+1\};\\
\dsig{0}  v_1 &= \delta v_1, & \delta \in \{-1,+1\};\\
\dtau v_1 &= \zeta^{\ell} v_1, & \zeta := e^{\pi i / m}, \ \ell\in \{0,1,\dots, 2m-1\}.
\end{align}
The main difference with the odd case is that $\dtau$ now has order $2m$ instead of order $m$.

Define $v_2 := \dsig{m} v_1$. We have
 \begin{align*}
\dtau v_2 &= \dtau\dsig{m}v_1 = \dsig{m}\dtau^{-1}v_1 = \zeta^{-\ell} v_2,&
\dsig{m}v_2 &= \dsig{m}^2v_1 = v_1,\\
zv_2 &= \dsig{m}zv_1 = z\dsig{m}v_1 = \varepsilon v_2,&
\dsig{0} v_2 &= \dsig{0}\dsig{m}v_1 = z\dsig{m}\dsig{0}v_1 = \varepsilon\delta v_2.
\end{align*}
So $\langle v_1, v_2\rangle \subset V$ is a submodule and by irreducibility of $V$ that means $V = \langle v_1, v_2\rangle$. We condition on whether $\zeta^{\ell} = \zeta^{-\ell}$. This only happens when $\ell\in\{0,m\}$.

When $\zeta^{\ell} \neq \zeta^{-\ell}$, then $\ell\in \{1,\dots m-1, m+1, \dots 2m-1\}$ and the matrices realising the actions of the elements in the basis $\{v_1,v_2\}$ are given by
\begin{equation}
\begin{aligned}
X(\dtau) &= \begin{pmatrix}
\zeta^{\ell} & 0 \\ 0 & \zeta^{-\ell}
\end{pmatrix}, & X(\dsig{m}) &= \begin{pmatrix}
0&1\\1&0
\end{pmatrix},&
X(z) &= 
\begin{pmatrix}
\varepsilon & 0 \\ 0&\varepsilon	
\end{pmatrix}, & X(\dsig{0}) &= \begin{pmatrix}
\delta & 0 \\ 0&\varepsilon\delta
\end{pmatrix}.
\end{aligned}
\end{equation}
Notice that the condition $\dtau^m = z$ restricts the possible values of $\varepsilon$ to $\varepsilon = (-1)^{\ell}$. Indeed, as $\ell$ is neither $0$ nor $m$ then $\dtau^m v_1 = \zeta^{m\ell} v_1 = (-1)^{\ell} v_1$, but also $\dtau^m v_1 = zv_1 = \varepsilon v_1$.


 We again demand that the imaginary part of $\zeta^{\ell}$ is positive and thus $\ell\in\{1, \dots, m-1\}$. There are $2p-1$ choices for $\ell$ and, for each $\ell$, 2 choices for $\delta$: a total of $4p-2$ representations.

 The cases following from $\zeta^{\ell} = \zeta^{-\ell}$ follow almost the same argument as the odd case. When $\ell =0$, 
  it gives $4$ one-dimensional non-equivalent representations. The actions are then $\dsig{m}v_1 = \beta v_1$ and $\dsig{0}v_1 = \delta v_1$ with $\beta,\delta\in \{-1,+1\}$. There is no possible representation when $\varepsilon = -1$ as $\dtau^mv_1 = v_1$ and $\dtau^m = z$.
 
 When $\ell = m$, 
  it gives $4$ one-dimensional non-equivalent representations with $\varepsilon =1$. Namely $\dsig{m}v_1 = \beta v_1$ and $\dsig{0}v_1 = \delta v_1$ with $\beta,\delta\in \{-1,+1\}$. There is again no possible representation when $\varepsilon = -1$ as $\dtau^mv_1 = v_1$ and $\dtau^m = z$.

In total, we get $4p-2$ two-dimensional non-equivalent irreducible representations from the first case and  $8$ one-dimensional non-equivalent irreducible representations from the second and third cases, for a total of $4p+6$, which is equal to the number of conjugacy classes. Hence, all of them have been found. Furthermore, indeed $(4p-2)\times 2^2 + 8\times 1^2 = 16p = |\dcover{W}^+|$.

The construction of the irreducible representations of the negative double covering $\dcover{W}^-$ follows in almost the same way, with only a slight difference on $\dsig{m}$. In the following theorem, we summarise our results.

\begin{theorem}\label{thm:irrepsW}
	Let $m$ be a positive integer and $W = \ZZ_2\times D_{2m}$. The complete sets of non-equivalent irreducible representations of the two double coverings $\dcover{W}^+$ and $\dcover{W}^-$ are given by the following tables.
	\begin{itemize}
		\item \textbf{(Odd $\mathbf{m = 2p+1}$).} For the positive double covering $\dcover{W}^+$, the $4p+5$ finite-dimensional non-equivalent irreducible representations are given by $4$ one-dimensional representations $X_i$ and $4p+1$ two-dimensional representations $Y_j = Y_j(\varepsilon,\delta)$ with actions given on generators $z$, $\dsig{0}$, $\dsig{m}$ and $\dtau:= \dsig{1}\dsig{m}$ by
		\begin{equation*}
			\begin{array}{r|cccc|cccccc}
			\dcover{W}^+	& \multicolumn{4}{c}{\text{one-dimensional}} & \multicolumn{6}{c}{\text{two-dimensional}}\\
				\hline
				& X_1 & X_2 & X_3 &X_4 & Y_0 & Y_1 & \cdots & Y_j & \cdots &  Y_p\\
				z& 1 & 1 & 1 &1 & \smallmat{-1&0\\0&-1} & \smallmat{\varepsilon & 0 \\ 0 & \varepsilon} & \cdots & \smallmat{\varepsilon & 0 \\ 0 & \varepsilon} & \cdots &\smallmat{\varepsilon & 0 \\ 0 &\varepsilon}\\
				\dsig{0} & 1 & -1 & 1 & -1 & \smallmat{1&0\\0&-1}  & \smallmat{\delta & 0 \\ 0& \varepsilon\delta} & \cdots &\smallmat{\delta & 0 \\ 0& \varepsilon\delta} & \cdots & \smallmat{\delta & 0 \\ 0& \varepsilon\delta}\\
				\dsig{m} & 1 & 1 & -1 & -1 & \smallmat{0&1\\1&0} & \smallmat{0&1\\1&0} & \cdots & \smallmat{0&1\\1&0} & \cdots &\smallmat{0&1\\1&0}\\
				\dtau & 1 & 1 &1 & 1 & \smallmat{1&0\\0&1} & \smallmat{\zeta^2 & 0 \\ 0&\zeta^{-2}} & \cdots & \smallmat{\zeta^{2j} & 0 \\ 0&\zeta^{-2j}} & \cdots & \smallmat{\zeta^{2p}&0\\0&\zeta^{-2p}}
			\end{array}
		\end{equation*}
		where $\delta$ and $\varepsilon$ take values in the set $\{-1,+1\}$ and $\zeta := e^{\pi i/m}$. 
		
		The $4p+5$ non-equivalent finite-dimensional irreducible representations of the negative covering $\dcover{W}^-$ are given by $4$ one-dimensional representations $X_i$ and $4p+1$ two-dimensional representations $Y_j = Y_j(\delta)$, with their actions on generators given in the next table 
		\begin{equation*}
		\begin{array}{r|cccc|cccccc}
		\dcover{W}^-	& \multicolumn{4}{c}{\text{one-dimensional}} & \multicolumn{6}{c}{\text{two-dimensional}}\\
		\hline
		& X_1 & X_2 & X_3 &X_4 & Y_m & Y_1 & \cdots & Y_j & \cdots &  Y_{2p}\\
		z& 1 & 1 & 1 &1 & \smallmat{-1&0\\0&-1} & \smallmat{-1 & 0 \\ 0 & -1} & \cdots & \smallmat{(-1)^j & 0 \\ 0 & (-1)^j} & \cdots &\smallmat{1 & 0 \\ 0 &1}\\
		\dsig{0} & 1 & -1 & 1 & -1 & \smallmat{i&0\\0&-i}  & \smallmat{\delta & 0 \\ 0& -\delta} & \cdots &\smallmat{\delta & 0 \\ 0& (-1)^j\delta} & \cdots & \smallmat{\delta & 0 \\ 0& \delta}\\
		\dsig{m} & 1 & 1 & -1 & -1 & \smallmat{0&-1\\1&0} & \smallmat{0&-1\\1&0} & \cdots & \smallmat{0&(-1)^j\\1&0} & \cdots &\smallmat{0&1\\1&0}\\
		\dtau & 1 & 1 &1 & 1 & \smallmat{-1&0\\0&-1} & \smallmat{\zeta & 0 \\ 0&\zeta^{-1}} & \cdots & \smallmat{\zeta^j & 0 \\ 0&\zeta^{-j}} & \cdots & \smallmat{\zeta^{2p}&0\\0&\zeta^{-2p}}
		\end{array}
		\end{equation*}
		where $\zeta := e^{\pi i/m}$ and $\delta$ takes value in $\{-1,+1\}$ if $j$ is even, and in $\{-i,+i\}$ if $j$ is odd.
		\item \textbf{(Even $\mathbf{m = 2p}$).}
		For the positive covering $\dcover{W}^+$, the $4p+6$ finite-dimensional non-equivalent irreducible representations are given by $8$ one-dimensional representations $X_i$ and $4p-2$ two-dimen\-sional representations $Y_j = Y_j(\delta)$ with actions given on generators $z,\dsig{0}, \dsig{m}$ and $\dtau:= \dsig{1}\dsig{m}$ by:
		\begin{equation*}
		\begin{array}{r|cccccccc|ccccc}
		\dcover{W}^+	& \multicolumn{8}{c}{\text{one-dimensional}} & \multicolumn{5}{c}{\text{two-dimensional}}\\
		\hline
		& X_1 & X_2 & X_3 &X_4 & X_5 & X_6 & X_7 & X_8 & Y_1 & \cdots & Y_j &\cdots& Y_{2p-1} \\
		z& 1 & 1 & 1 &1 &1&1&1&1 & \smallmat{-1 &0 \\ 0&-1} &\cdots & \smallmat{(-1)^j &0 \\ 0&(-1)^j}& \cdots &\smallmat{-1 &0 \\ 0&-1} \\ 
		\dsig{0} & 1 & -1 & 1 & -1 & 1 & -1 & 1 & -1 & \smallmat{\delta &0 \\ 0 &-\delta} &\cdots& \smallmat{\delta &0 \\ 0 &(-1)^j\delta} &\cdots& \smallmat{\delta &0 \\ 0 &-\delta} \\
		\dsig{m} & 1 & 1 & -1 & -1 & 1 & 1 & -1 & -1 &\smallmat{0&1\\1&0} &\cdots& \smallmat{0&1\\1&0} &\cdots&\smallmat{0&1\\1&0}\\
		\dtau & 1 & 1 &1 & 1 & -1 & -1 &-1 & -1 &\smallmat{\zeta &0 \\ 0& \zeta^{-1}} &\cdots& \smallmat{\zeta^j &0 \\ 0& \zeta^{-j}} &\cdots& \smallmat{\zeta^{2p-1} &0 \\ 0& \zeta^{-(2p-1)}}
		\end{array}
		\end{equation*}
		where $\delta \in \{-1,+1\}$ and $\zeta := e^{\pi i/m}$. 
		
		The $4p+6$ finite-dimensional non-equivalent irreducible representations of $\dcover{W}^-$ are given by $8$ one-dimensional representations $X_i$ and $4p-2$ two-dimensional representations $Y_j=Y_j(\delta)$ presented by their actions on generators in the next table
		\begin{equation*}
		\begin{array}{r|cccccccc|ccccc}
		\dcover{W}^-	& \multicolumn{8}{c}{\text{one-dimensional}} & \multicolumn{5}{c}{\text{two-dimensional}}\\
		\hline
		& X_1 & X_2 & X_3 &X_4 & X_5 & X_6 & X_7 & X_8 & Y_1 &\cdots&Y_j &\cdots& Y_{2p-1} \\
		z& 1 & 1 & 1 &1 &1&1&1&1 & \smallmat{-1 &0 \\ 0&-1}&\cdots& \smallmat{(-1)^j &0 \\ 0&(-1)^j} &\cdots&\smallmat{-1 &0 \\ 0&-1} \\ 
		\dsig{0} & 1 & -1 & 1 & -1 & 1 & -1 & 1 & -1 & \smallmat{\delta &0 \\ 0 &-\delta} &\cdots& \smallmat{\delta &0 \\ 0 &(-1)^j\delta} &\cdots& \smallmat{\delta &0 \\ 0 &-\delta} \\
		\dsig{m} & 1 & 1 & -1 & -1 & 1 & 1 & -1 & -1 &\smallmat{0&-1\\1&0} &\cdots& \smallmat{0&(-1)^j\\1&0} &\cdots&\smallmat{0&-1\\1&0}\\
		\dtau & 1 & 1 &1 & 1 & -1 & -1 &-1 & -1 &\smallmat{\zeta &0 \\ 0& \zeta^{-1}} &\cdots& \smallmat{\zeta^j &0 \\ 0& \zeta^{-j}} &\cdots& \smallmat{\zeta^{2p-1} &0 \\ 0& \zeta^{-(2p-1)}}
		\end{array}
		\end{equation*}
		where $\zeta := e^{\pi i /m}$ and $\delta \in \{-1,+1\}$ if $j$ is even, and $\delta \in \{-i,+i\}$ if $j$ is odd.
	\end{itemize}
\end{theorem}
\end{document}